\documentclass[11pt,a4paper]{article}

\usepackage{array,amsbsy,amscd,amsfonts,color,amssymb,amstext,amsmath,latexsym,amsthm,amscd
,multicol,graphicx,makeidx,lineno}
\usepackage[english]{babel}
\usepackage{hyperref}

\numberwithin{equation}{section}

\makeindex

\newtheorem{theorem}{Theorem}
\newtheorem{proposition}[theorem]{Proposition}

\newtheorem{lemma}[theorem]{Lemma}

\numberwithin{theorem}{section}

\theoremstyle{definition}
\theoremstyle{definition}\newtheorem{definition}[theorem]{Definition}
\theoremstyle{definition}\newtheorem{remark}[theorem]{Remark}
\theoremstyle{definition}\newtheorem{example}[theorem]{Example}
\theoremstyle{definition}\newtheorem{construction}[theorem]{Construction}
\theoremstyle{definition}
\theoremstyle{definition}
\def\proofof [#1] {\noindent {\bf Proof of #1. } }

\def\al #1.{{\mathcal{#1}}}

\renewcommand{\AA}{\mathfrak{A}}

\newcommand{\A}{\mathcal{A}}
\newcommand{\U}{\mathcal{U}}
\newcommand{\RR}{\mathcal{R}}
\newcommand{\B}{\mathcal{B}}

\newcommand{\I}{\mathcal{I}}
\newcommand{\K}{\mathcal{K}}
\renewcommand{\H}{\mathcal{H}}

\newcommand{\ZZ}{\mathcal{Z}}
\newcommand{\F}{\mathcal{F}}

\newcommand{\N}{\mathbb{N}}
\newcommand{\NN}{\mathbb{N}_0}
\newcommand{\R}{\mathbb{R}}
\newcommand{\C}{\mathbb{C}}

\renewcommand{\S}{{S^1}}

\newcommand{\Z}{\mathbb{Z}}

\newcommand{\unit}{\mathbf{1}}

\newcommand{\bp}{\begin{proof}}
\newcommand{\ep}{\end{proof}}
\newcommand{\bdp}{\begin{dproof}}
\newcommand{\edp}{\end{dproof}}
\newcommand{\ra}{\rightarrow}

\newcommand{\Uone}{\operatorname{U}(1)}

\newcommand{\Loop}{\operatorname{L}}
\newcommand{\Diff}{\operatorname{Diff}(\S)}

\newcommand{\FF}{\mathfrak{f}}

\newcommand{\PSL}{\operatorname{PSL(2,\R)}}

\newcommand{\PSI}{\operatorname{PSL(2,\R)}^{(\infty)}}

\newcommand{\Mat}{\operatorname{M}}
\newcommand{\GL}{\operatorname{GL}}

\newcommand{\SVir}{\operatorname{SVir}}
\newcommand{\SVirc}{\operatorname{SVir}_c}

\newcommand{\tr}{\operatorname{tr }}
\newcommand{\Ad}{\operatorname{Ad }}

\newcommand{\rmd}{\operatorname{d}}
\newcommand{\rmi}{\operatorname{i}}

\newcommand{\rme}{\operatorname{e}}
\newcommand{\lw}{\operatorname{lw}}

\newcommand{\eps}{\varepsilon}
\newcommand{\de}{\delta}

\newcommand{\pR}{{\pi_R}}

\newcommand{\supp}{\operatorname{supp }}

\newcommand{\dom}{\operatorname{dom }}

\renewcommand{\ker}{\operatorname{ker }}

\renewcommand{\dim}{\operatorname{dim }}

\newcommand{\Cci}{C^{\infty}}

\newcommand{\boxy}{\hfill $\Box$\vspace{5mm}}

\newcommand{\End}{\operatorname{End}}

\newcommand{\diag}{\operatorname{diag}}
\newcommand{\ind}{\operatorname{ind }}

\renewcommand{\lg}{\mathfrak{g}}

\newcommand{\id}{\operatorname{id}}

\newcommand{\ie}{{i.e.,\/}\ }

\newcommand{\eg}{{e.g.\/}\ }
\newcommand{\cf}{{cf.\/}\ }

\title{Superconformal Nets and Noncommutative Geometry}
\author{
{\sc Sebastiano Carpi}\footnote{Supported in part by the ERC Advanced Grant 227458 OACFT ``Operator Algebras and Conformal
Field Theory", PRIN-MIUR, GNAMPA-INDAM and the EU network ``Noncommutative Geometry".}\\
\small Dipartimento di Economia,
Universit\`a di Chieti-Pescara ``G. d'Annunzio''\\
\small Viale Pindaro, 42, I-65127 Pescara, Italy\\
\small E-mail: {\tt s.carpi@unich.it}\\
\phantom{X}\\
{\sc Robin Hillier}$^*$\footnote{Marie-Curie Fellow of the Istituto Nazionale di Alta Matematica. Formerly supported by a one-year starting scholarship of the Gottlieb Daimler- und Karl Benz-Stiftung.}\\
\small School of Mathematics,
Cardiff University\\
\small Senghennydd Road, Cardiff CF24 4AG, United Kingdom\\
\small E-mail: {\tt hillierro@cardiff.ac.uk}\\
\phantom{X}\\
{\sc Roberto Longo}$^*$\\
\small Dipartimento di Matematica,
Universit\`a di Roma ``Tor Vergata'',\\
\small Via della Ricerca Scientifica, 1, I-00133 Roma, Italy\\
\small E-mail: {\tt longo@mat.uniroma2.it}
\phantom{X}\\
}

\date{2 August 2013}

\begin{document}


\maketitle

\begin{abstract}
This paper provides a further step in the program of studying superconformal nets over $\S$ from the point of view of noncommutative geometry. For any such net $\A$ and any family $\Delta$ of localized endomorphisms of the even part $\A^\gamma$ of $\A$, we define the locally convex differentiable algebra $\AA_\Delta$ with respect to a natural Dirac operator coming from supersymmetry. Having determined its structure and properties, we study the family of spectral triples and JLO entire cyclic cocycles associated to elements in $\Delta$ and show that they are nontrivial and that the cohomology classes of the cocycles corresponding to inequivalent endomorphisms can be separated through their even or odd index pairing with K-theory in various cases. We illustrate some of those cases in detail with superconformal nets associated to well-known CFT models, namely super-current algebra nets and super-Virasoro nets. All in all, the result allows us to encode parts of the representation theory of the net in terms of noncommutative geometry. 
\end{abstract}

\vspace{8\baselineskip}

\pagebreak

\tableofcontents

\section{Introduction}

In this article we would like to explore certain aspects of conformal quantum field theory that emerge by looking at a superconformal net from the point of view of noncommutative geometry. By such a connection between two important areas of mathematics, we hope to gain new insight into conformal field theory. In order to make this article reasonably self-contained, we provide almost all the necessary basics in (super-)conformal field theory and noncommutative geometry.

By conformal field theory \cite{DMS96} we actually mean here chiral conformal quantum field theory on the unit circle $\S$ (i.e., generated by fields depending on one light-ray coordinate only) and we work within the operator algebraic approach to quantum field theory \cite{Haag} in its chiral conformal field theory version, see e.g. \cite{BGL93,BMT88,FJ96,FG,GL,KL04,Lo89,Lon90,Was,Xu07} for some representative works in this and related settings. 

In the local case, the basic object is a local conformal net: a net of von Neumann algebras, indexed by the proper open intervals of $\S$, satisfying a set of natural axioms. Such a net contains always a Virasoro net (the net associated to the unitary vacuum representation of the Virasoro Lie algebra with given central charge) as a minimal conformal subnet. Then we call a net superconformal if it satisfies a certain graded-local version of those axioms and contains a super-Virasoro net introduced in Example \ref{ex:superVir} (a net associated to the ($N=1$) super-Virasoro algebra, a graded extension of the Virasoro algebra). 

By noncommutative geometry we mean here the operator algebraic extension of differential geometry according to Connes \cite{Co-NCDG,Co1}. The basic objects are spectral triples $(A,\H_\pi,D)$, whose three components generalize the algebra of smooth functions on a manifold, the left representation on its spinor bundle, and the Dirac operator on that bundle, respectively. 
The crucial point is that spectral triples give rise to certain Chern characters in cyclic cohomology (the JLO cocycles), dual generalizations of the Chern characters in de Rham cohomology, which then in turn pair with $K$-theory: thus one can compute numbers characterizing geometric structures.

Concerning our motivations, we mention that crucial structural objects for a given net are its sectors, the unitary equivalence classes of its irreducible representations (cf. \cite{DHR}). One can envisage that the sectors are to be the basic ingredients for an index theorem in the quantum infinite-dimensional case (\cf \cite{Lo01}). Moreover, by considering Weyl's asymptotic expansion of the Laplace operator on a compact manifold, it has been shown in \cite{KL06} that the conformal Hamiltonian of a modular conformal net in a given sector has a similar meaning as the Laplace operator on an infinite-dimensional manifold. Since in the commutative setting the Dirac operator is an odd square-root of the Laplacian, it suggests itself to look at graded-local conformal nets whose conformal Hamiltonian in certain representations has an odd square-root. This led to consider superconformal nets and in particular Ramond representations of such nets \cite{CKL}.

In this light, we would like to study spectral triples $(\AA,(\pi,\H_\pi),Q)$ in superconformal quantum field theory. In \cite{CHKL}, together with Y. Kawahigashi, we dealt with (nets of) graded spectral triples associated to the unitary representations with positive energy of the 
Ramond super-Virasoro algebra. In the present article we  define a different but related version of spectral triples, which shall be one of the objectives of Section \ref{sec:gen-ST}. Once a ($\theta$-summable) spectral triple $(\AA,(\pi,\H_\pi),Q)$ for a given superconformal net is fixed, we obtain an entire cyclic cocycle, called JLO cocycle. Changing now the representation $\pi$ of our superconformal net changes the spectral triple. But does the cohomology class of the cocycle also change? If we have no abstract reasoning available to answer this question, we could try to find suitable K-classes which separate the cocycles corresponding to the several representations: since they are dual objects, there is a natural integer-valued index pairing, and so we should be able to do explicit computations. Summing up, we have in mind the following association:
\begin{center}
\begin{tabular}{m{4.5cm}m{0.5cm}m{4.5cm}m{0.5cm}m{4cm}}
\qquad \qquad \textbf{CFT}  && \qquad \qquad \textbf{NCG} &&\quad \textbf{Number} \\
 &&&&\\
$ \begin{array}{c} \mbox{superconformal net}\\ \mbox{(and its representations)} \end{array} $
 &$\mapsto$ & $\left\langle\begin{array}{c} \mbox{ entire cyclic cocycles}\\ \mbox{ K-classes} \end{array} \right\rangle$ &$\mapsto$ & $\Z$-valued pairing
\end{tabular}
\end{center}
If this procedure gives a non-trivial result, then we can express information about the original superconformal net in terms of noncommutative geometry. Moreover, in later steps we may study the whole extent of this relationship and a possible inversion of the association. 

We would like to mention that the present approach will be somehow differential-geometric involving some differentiability and admissibility conditions on the representations of our net. A general and purely topological approach, dispensing with supersymmetry, has been recently achieved in \cite{CCHW,CCH} for completely rational local conformal nets. We expect a deeper relation to the present work in the case where the completely rational net is the even part of a superconformal net. 

Let us now briefly explain how we are going to put our plan into practice in Sections \ref{sec:gen-ST} and \ref{sec:gen-pairing}. 
Sections \ref{sec:CFT} and \ref{sec:NCG} contain the necessary preliminaries on superconformal nets and noncommutative geometry. 
Since they are to a large extent collections of known facts, included in order to keep this article reasonably self-contained, we provide proofs only for the new results there, while referring to literature otherwise. 

Given a superconformal net $\A$, we shall fix an irreducible Ramond representation $(\pR,\H_R)$. It will be either graded or ungraded. In such a representation, there exists automatically a square-root of the conformal Hamiltonian $L_0^\pR$ up to an additive constant, namely $Q=G_0^\pR$ coming from the super-Virasoro algebra, and it is odd in case the representation is graded. This gives rise to a derivation $\delta$ on $B(\H_R)$, and usually satisfies the condition for $\theta$-summability: $\rme^{-tQ^2}$ is trace-class on $\H_R$, for all $t>0$. Then the JLO formula defines an entire cyclic cocycle over the even subalgebra, which gives rise to a pairing in K-theory. Actually, we have a family of cocycles since we may perform this construction for every local algebra $\A(I)$, $I\in\I$, as well as for nice global algebras like the universal C*- or von Neumann algebra. While we kept this generality in \cite{CHKL}, we shall recognize below that the local algebras are not sufficient for our task and we have to choose a global one in Definition \ref{def:gen2-AA}. 

As is known, locally normal localized representations of the even subnet $\A^\gamma$ correspond to localized transportable (DHR) endomorphisms of the universal C*-algebra $C^*(\A^\gamma)$ or its enveloping von Neumann algebra $W^*(\A^\gamma)$. 

Given a family of localized covariant endomorphisms $\Delta$ we consider the largest  subalgebra $\AA_\Delta\subset W^*(\A^\gamma) \cap \dom(\delta)$ such that $(\AA_\Delta, (\pR \circ \rho ,\H_R), Q)$ is a spectral triple for all $\rho \in \Delta$, \cf Definition \ref{def:gen2-AA}. 
There is a natural locally convex topology on $\AA_\Delta$ which guarantees that the JLO cocycle $\tau_\rho$ associated to the $\theta$-summable spectral triple 
$(\AA_\Delta, (\pR \circ \rho ,\H_R), Q)$ is entire for all $\rho \in \Delta$. This way we will end up with a family of entire cyclic cocycles 
$(\tau_\rho)_{\rho\in\Delta}$ corresponding to the family of localized endomorphisms $\Delta$: geometric quantities associated to quantum field theoretical ones in a non-trivial way. Imposing further optional conditions on the set $\Delta$ like differentiable transportability as in Definition \ref{def:gen2-Delta} results in a very rich structure and several stability properties of $\AA_\Delta$, the spectral triples and the cocycles.

Now the above JLO cocycles $\tau_\rho$ might be all cohomologous. That this is actually not the case, for suitably chosen $\Delta$, can be proved by pairing the family of cocycles with a suitable family of $K_*(\AA_\Delta)$-classes, represented by idempotent or invertible elements in the case of even or odd spectral triples, respectively. The right representatives of these classes (or at least one possible solution) will be constructed in Section \ref{sec:gen-pairing}. They are related to certain finite-dimensional subprojections of the positive eigenspace $\H_{R,+}$ of the grading unitary or to certain shift unitaries on the eigenspaces of $Q$, respectively.
With these two families at hand -- the cocycles and K-classes corresponding to the representations -- we then obtain a well-defined index pairing between them, separating the (JLO cocycles corresponding to the) equivalence classes of representations in $\Delta$ as described in Theorem \ref{th:gen2-pairing-even} and \ref{th:gen2-pairing-odd}. 

After these general investigations and constructions, we shall apply them in Section \ref{sec:loop} to important models of superconformal nets: super-current algebra nets and the super-Virasoro net. Our goal will be to show that our assumptions and conditions make sense, to understand the geometric, algebraic, and physical meaning of the involved objects better, and to see how far we can go with our correspondence between superconformal nets and noncommutative geometry. This way, our  work becomes self-contained and complete, but offering many potential interactions with related issues. 

One of those issues is the study of higher degree of supersymmetry, \ie super-Virasoro algebras involving further odd fields apart from $G$. This can be done for arbitrary degree, but the first and already very interesting case with new emerging structures is the $N=2$ super-Virasoro algebra (in contrast to the usual $N=1$ super-Virasoro algebra investigated in the present paper and in \cite{CHKL,CKL}). The corresponding nets, their representations and extensions were studied by us together with Y. Kawahigashi and F. Xu in \cite{CHKLX}. The noncommutative differential-geometric aspects and the resulting index pairing there are similar to the present ones, but with some important differences. Other interesting future directions could be related to the results in \cite{BCL,CN}.

This work is based in part on RH's PhD thesis at Universit\`a di Roma``Tor Vergata" \cite{PhD-Thesis}.

\bigskip

Throughout this article all (associative) $*$-algebras and all $*$-representations are assumed to be unital. Nevertheless, sometimes we will repeat these assumptions simply in order to underline them.

\section{Superconformal nets}\label{sec:CFT}

We provide here a brief summary on (graded-)local conformal nets, see also \cite{CKL} and the references here below. 

Let $\S = \{ z\in \C: |z|=1 \}$ be the unit circle, let $\Diff$ be the infinite-dimensional (real) Lie group of orientation preserving smooth diffeomorphisms of $\S$ and denote by $\Diff^{(n)}$, $n\in \N \cup \{ \infty \}$, the corresponding $n$-cover. In particular $\Diff^{(\infty)}$ is the universal covering group of $\Diff$. The group $\PSL$ of 
M\"{o}bius transformations of $\S$ is a three-dimensional subgroup of $\Diff$. We denote by 
$\PSL^{(n)} \subset \Diff^{(n)}$, $n\in \N \cup \{ \infty \}$, the corresponding $n$-cover so that $\PSI$ is the universal covering group of $\PSL$. We denote by  $\dot{g} \in \Diff$ the image of $g \in \Diff^{(\infty)}$ under the covering 
map. Since the latter restricts to the covering map of $\PSI$ onto $\PSL$ we have $\dot{g} \in \PSL$ for all 
$g \in \PSI$.   

Now let $\I$ denote the set of nonempty and non-dense open intervals of $\S$. For any $I\in \I$, $I'$ denotes 
the interior of $\S \setminus I$. We write $\Cci(\S):= \Cci (\S, \R)$ for the smooth real-valued functions on $\S$ and, given $I\in\I$, $\Cci(\S)_I$ for the subspace of those with support in $I$. The subgroup $\Diff_I\subset\Diff$ of diffeomorphisms localized in $I$ is defined 
as the stabilizer of $I'$ in $\Diff$ namely the subgroup of $\Diff$ whose elements are the diffeomorphisms acting trivially 
on $I'$. Then, for any $n \in \N \cup \{ \infty \}$, $\Diff^{(n)}_I$ denotes the connected component of the identity 
of the pre-image of $\Diff_I$ in $\Diff^{(n)}$ under the covering map. Then we write $\I^{(n)}$ for the set of intervals in $\S^{(n)}$ which map to an element in $\I$ under the covering map. Moreover, we often identify $\R$ with $\S\setminus\{-1\}$ by means of the Cayley transform, and we write $\I_\R$ (or $\bar{\I}_\R$) for the set of bounded open intervals (or bounded open intervals and open half-lines, respectively) in $\R$. After the above identification 
of $\R$ with $\S\setminus\{-1\}$ we have the inclusions $\I_\R \subset \bar{\I}_\R \subset \I$.
\begin{definition}\label{def:CFT-net}
 A \emph{graded-local conformal net $\A$ on $\S$} is a map $I \mapsto \A(I)$ from the set of intervals $\I$ to the set of von Neumann algebras acting on a common infinite-dimensional separable Hilbert space $\H$ which have the 
following properties:
\begin{itemize}
 \item[$(A)$] \emph{Isotony.} $\A(I_1)\subset \A(I_2)$ if $I_1,I_2\in\I$ and $I_1\subset I_2$.
\item[$(B)$] \emph{M\"{o}bius covariance.} There is a strongly continuous unitary representation $U$ of $\PSI$ such that
\[
 U(g)\A(I)U(g)^* = \A(\dot{g}I), \quad g\in \PSI , I\in\I .
\]
\item[$(C)$] \emph{Positive energy.} The conformal Hamiltonian $L_0$ (\ie the self-adjoint generator of the
restriction of $U$ to the lift to $\PSI$ of the one-parameter anti-clockwise rotation subgroup of $\PSL$) is positive.
 \item[$(D)$] \emph{Existence and uniqueness of the vacuum.} There exists a $U$-invariant vector $\Omega\in\H$ which is unique up to a phase and cyclic for $\bigvee_{I\in\I} \A(I)$.
 \item[$(E)$] \emph{Graded locality.} There exists a self-adjoint unitary $\Gamma$ (the grading unitary) on $\H$ satisfying 
 $\Gamma \A(I) \Gamma = \A(I)$ for all $I\in \I$ and $\Gamma \Omega =\Omega$ and such that
 \[
 \A(I')\subset Z \A(I)'Z^*,\quad I\in\I ,
\]
where
\[
 Z:=\frac{\unit - \rmi\Gamma}{1-\rmi}.
\]
\item[$(F)$] \emph{Diffeomorphism covariance.} There is a strongly continuous projective unitary representation of 
$\Diff^{(\infty)}$, denoted again by $U$, extending the unitary representation of $\PSI$ and such that
\[
 U(g)\A(I)U(g)^* = \A(\dot{g}I), \quad g\in \Diff^{(\infty)} , I\in\I,
\]
and
\[
 U(g) x U(g)^* = x,\quad x\in \A(I'), g\in \Diff^{(\infty)}_I, I\in\I.
\]
\end{itemize}
A \emph{local conformal net} is a graded-local conformal net with trivial grading $\Gamma=\unit$. The \emph{even subnet}\index{even subnet} of a graded-local conformal net $\A$ is defined as the fixed point subnet $\A^\gamma$, with grading gauge automorphism $\gamma=\Ad\Gamma$. It can be shown that the projective representation $U$ of $\Diff^{\infty}$ commutes with $\Gamma$, cf. 
\cite[Lemma 10]{CKL}. Accordingly the restriction of  $\A^\gamma$ to the even subspace $\H^\Gamma$ of $\H$ is a local conformal net with respect to the restriction to this subspace of the projective representation $U$ of $\Diff^{\infty}$. This local conformal net will again be denoted by $\A^\gamma$ while the corresponding representations of $\PSI$ together with its extension to $\Diff^{\infty}$ will be denoted by $U^\gamma$. 
\end{definition}

Note that graded-local conformal nets on $\S$ are called Fermi conformal nets in \cite{CKL}. A map $I \mapsto \A(I)$ 
satisfying all the properties in the above definition with the possible exception of  $(F)$ (diffeomorphism covariance) is called a \emph{graded-local M\"{o}bius covariant net on $\S$} (or M\"{o}bius covariant Fermi net on $\S$). Some results of this paper could be formulated in terms of M\"{o}bius covariant nets but for simplicity of the exposition we will always consider diffeomorphism covariant nets. Actually we will restrict ourselves mainly to the class of  graded-local nets admitting a supersymmetric extension of the diffeomorphism symmetry namely the class of superconformal nets on $\S$, defined below. When we want to permit both situations, we shall denote by $G$ either of the two groups $\PSL$ or $\Diff$. For the rest of this section, we write $\A$ for a generic graded-local conformal net over $\S$ acting on the Hilbert space $\H$ and with grading automorphism $\gamma$ as in the above definition.

Some of the consequences  \cite{CKL,CW05,FG,FJ96,GL} of the preceding definition are:

\begin{itemize}
\item[$(1)$] \emph{Reeh-Schlieder Property.} $\Omega$ is cyclic and separating for every $\A(I)$, $I\in\I$.
\item[$(2)$] \emph{Bisognano-Wichmann Property.} Let $I\in \I$ and let $\Delta_I$, $J_I$ be the modular operator and the modular conjugation of $\left(\A(I), \Omega \right)$.  Then we have 
\[
 U(\delta_I(-2\pi t)) = \Delta_I^{\rmi t}, \quad t\in \R .
\]

Moreover the unitary representation $U:\PSI \mapsto B(\H)$ extends to an (anti-)\linebreak unitary representation of  
$\PSI \rtimes \Z/2$  determined by 
\[
U(r_I)=Z J_I 
\]
and acting covariantly on $\A$.
Here $(\delta_I(t))_{t\in\R}$ is (the lift to $\PSI$ of) the one-parameter subgroup of dilations with respect to $I$ and $r_I$ the point reflection of the interval $I$ onto the complement $I'$. $r_I$ is identified with $1\in \Z/2$ and the corresponding automorphism $g\mapsto r_Igr_I$ of $\PSI$ is determined by the requirement that the image of $r_Igr_I$ in $\PSL$ under the covering map is equal to $r_I\dot{g}r_I$ for all 
$g\in \PSI$, cf. \cite{GL1,GL}. 

\item[$(3)$] \emph{Graded Haag Duality.} $\A(I')=Z\A(I)'Z^*$, for $I\in\I$.
\item[$(4)$] \emph{Outer regularity.} 
\[
 \A(I_0) = \bigcap_{I\in\I,I\supset \bar{I_0}} \A(I),\quad I_0\in\I.
\]
\item[$(5)$] \emph{Additivity.} If $I = \bigcup_{\alpha} I_\alpha$ with $I, I_\alpha \in \I$ a certain family, then 
$\A(I) = \bigvee_\alpha \A(I_\alpha)$.
 \item[$(6)$] \emph{Factoriality.} $\A(I)$ is a type $III_1$-factor, for $I\in\I$.
 \item[$(7)$] \emph{Irreducibility.} $\bigvee_{I\in\I} \A(I)= B(\H)$.
 \item[$(8)$] \emph{Vacuum Spin-Statistics theorem.} $\rme^{\rmi 2\pi L_0} =\Gamma$, in particular $\rme^{\rmi 2\pi L_0}=\unit$ for local nets, where $L_0$ is the infinitesimal generator from above corresponding to rotations. Hence the representation $U$ of $\PSI$ factors through a representation of $\PSL^{(2)}$ ($\PSL$ in the local case) and consequently its extension $\Diff^{(\infty)}$ factors through a projective representation of $\Diff^{(2)}$  ($\Diff$ in the local case). 
 \item[$(9)$] \emph{Uniqueness of Covariance.} For fixed $\Omega$, the strongly continuous projective representation $U$ of $\Diff^{(\infty)}$ making the net covariant is unique.
\end{itemize}

\begin{definition}\label{def:CFT-split}
A graded-local conformal net $\A$ satisfies the \emph{split property} if, given $I_1,I_2\in\I$ such that $\bar{I}_1\subset I_2$, there is a type I factor $F$ such that
\[
 \A(I_1) \subset F \subset \A(I_2).
\]
\end{definition}

\begin{definition}\label{def:CFT-reps1}
A \emph{(DHR) representation} of $\A$ is a family $\pi=(\pi_I)_{I\in\I}$ of $*$-representations
\[
\pi_I : \A(I) \ra B(\H_\pi),\quad I\in\I,
\]
on a common Hilbert space $\H_\pi$ such that $\pi_{I_2}|_{\A(I_1)} = \pi_{I_1}$ whenever $I_1\subset I_2$.
\begin{itemize}
\item[-] $\pi$ is called \emph{locally normal} if every $\pi_I$ is normal.
\item[-] $\pi$ is called \emph{$G$-covariant} if there exists a projective unitary representation $U_\pi$ of $G^{\infty}$ on $\H_\pi$ satisfying
\[
U_\pi(g) \pi_I(x) U_\pi(g)^* = \pi_{\dot{g}I}(U(g)xU(g)^*),\quad g\in G^{\infty}, x\in\A(I),I\in\I.
\]
Here $G= \PSL$ or $G=\Diff$.
\item[-] $\pi$ has \emph{positive energy}  if it is $G$-covariant and the infinitesimal generator of the lift of the rotation subgroup in $U_\pi(G^{(\infty)})$ is positive. 
\item[-] We say that the operator $T\in B(\H_{\pi_1},\H_{\pi_2})$ intertwines two representations $\pi_1,\pi_2$ , if for every $I\in\I$, it intertwines $\pi_{1,I}$ and $\pi_{2,I}$. Two representations $\pi_1,\pi_2$ are \emph{unitarily equivalent} if they admit a unitary intertwiner. 
The unitary equivalence class of a representation $\pi$ is denoted by $[\pi]$. $\pi$ is said to be irreducible if its self-intertwiners coincide with the scalar multiples of the identity operator. The direct sum $\pi_1 \oplus \pi_2$ is defined by  $(\pi_1 \oplus \pi_2)_I := \pi_{1,I} \oplus \pi_{2,I}$, $I\in I$. Accordingly a representation $\pi$ of the net $\A$ is irreducible if and only if it is not unitarily equivalent to a direct sum 
of non-simultaneously zero representations.
The unitary equivalence classes of irreducible locally normal representations are called the \emph{sectors} of $\A$. 
\item[-] $\pi$ is called \emph{localized} in a certain interval $I_0\in\I$ if $\H_\pi=\H$ and, for every $I\in\I$ with $I\subset I_0'$, we have $\pi_I=\iota_I$. $\pi$ is said to be  \emph{localizable} in $I\in \I$ if is unitarily equivalent to a representation which is localized in $I$. 
\end{itemize}
\end{definition}

The identity representation $\pi_0$ of $\A$ on $\H$ is called the \emph{vacuum representation}, and it is automatically locally normal, $G$-covariant, and localized in any given interval. Moreover, if $\Gamma$ is non-trivial we will denote the vacuum representation of 
$\A^\gamma$ on $\H^\Gamma \subset \H$ by $\pi^\gamma_0$.  

In the above definition, note that when $G=\PSL$, the projective representation $U_\pi$ comes from a unique unitary representation, which we will denote by the same symbol. Hence, the generators of one-parameter subgroups of $U_\pi$ are uniquely determined, in particular, this is the case for the generator $L_0^\pi$ of rotations (conformal Hamiltonian). When $G=\Diff$, these generators are actually obtained by the unitary representation (also denoted by $U_\pi$) corresponding to the restriction of $U_\pi$ to $\PSI$.

The space $\H_\pi$ is separable if $\pi$ is locally normal and cyclic, and $\pi$ is localizable in every $I_0 \in \I$ and hence locally normal if $\H_\pi$ is separable, see e.g. \cite[App.B]{KLM}. If $\pi$ is locally normal, then it is automatically $\PSL$-covariant \cite{CKL,DFK} and of positive energy \cite{Wei06}. Moreover, the representation $U_\pi$ can be uniquely chosen to be inner, \ie such that $U_\pi(g)\in \bigvee_{I\in\I} \pi_I(\A(I))$, for all $g\in \PSL^{(\infty)}$ (\cf also \cite{Koe}). In the following, unless stated otherwise, $U_\pi$ will always denote this unique inner representation.

In our index pairing below the even subnet (which is a local conformal net) will play a central role, so let us collect here a few general facts about local conformal nets. Let us denote by $\B$ a generic local conformal net with vacuum representation $\pi_0$. 
In this case, if $\pi$ is a representation of $\B$ localized in $I_0$, then by Haag duality we have $\pi_I(\B(I)) \subset \B(I)$, for all $I\in \I$ containing $I_0$,   i.e., $\pi_I$ is an endomorphism of $\B(I)$, and we say that $\pi$ is a \emph{localized endomorphism} or \emph{DHR endomorphism} (localized in $I_0$) of the net $\B$. If $\pi_{I_0}(\B(I_0)) = \B(I_0)$ then $\pi_{I}(\B(I)) = \B(I)$ 
for all $I\in \I$ containing $I_0$ and we say that $\pi$ is a \emph{localized automorphism} of the net $\B$. For an analogous statement in the graded-local case we refer to \cite[Prop.14]{CKL}. 

If $\pi$ is a representation of $\B$ and $I_1,I_2 \in \I$ are disjoint intervals with $I_1 \neq I_2'$ then it follows from locality that 
$\pi_{I_1}(\B(I_1)) \subset \pi_{I_2}(\B(I_2))'$. Hence, if $\pi$ is locally normal, we have an inclusion of type III  factors 
$\pi_{I}(\B(I)) \subset \pi_{I'}(\B(I'))'$ for every $I\in \I$ as a consequence of additivity.  Moreover, it follows from the covariance of $\pi$ that 
the  minimal index $[ \pi_{I'}(\B(I'))': \pi_I(\B(I))]$ is independent of $I$. Its square root is called the \emph{statistical dimension} of the locally normal representation $\pi$, is denoted by $d(\pi)$ and depends only on the unitary equivalence class $[\pi]$ of $\pi$. 
If ${\pi_1}, {\pi_2}$ are locally normal then $d(\pi_1 \oplus \pi_2) =d(\pi_1) +d(\pi_2)$. If the representation $\pi$ is localized in $I_0\in \I$ and if $I\in \I$ contains $I_0$ we have  $d(\pi)= [\B(I):\pi_I(\B(I))]^{\frac12}$, i.e., $d(\pi)^2$ is the index of the unital endomorphism $\pi_I \in \End (\B(I))$. Accordingly, recalling that normal endomorphisms of von Neumann factors are always injective, the localized endomorphism $\pi$ of the net $\B$ is a localized automorphism if and only if $d(\pi)=1$. In general we have $d(\pi) \geq 1$.

We start with an important global algebra associated to it, introduced in \cite[Sect.2]{Fred2} and \cite[(5.1.7)]{FRS2} 
(see also \cite[Sect.8]{GL1}):

\begin{definition}\label{def:CFT-univC}
The \emph{universal C*-algebra} $C^*(\B)$ is the C*-algebra such that
\begin{itemize}
\item[-] for every $I\in\I$, there are unital embeddings $\iota_I:\B(I)\ra C^*(\B)$, such that $\iota_{I_1|\B(I_2)} = \iota_{I_2}$ whenever $I_1\subset I_2$, and $\iota_I(\B(I))$ generate $C^*(\B)$ as a C*-algebra;
\item[-] for every representation $\pi$ of $\B$ on some Hilbert space $\H_\pi$, there is a unique $*$-representation $\hat{\pi}:C^*(\B)\ra B(\H_\pi)$ such that
\[
\pi_I = \hat{\pi}\circ \iota_I,\quad I\in \I.
\]
\end{itemize}
The universal C*-algebra can be shown to be unique up to isomorphism. Let $(\hat{\pi}_u,\H_u)$ be its \emph{universal representation}: the direct sum of all GNS representations $\hat{\pi}$ of $C^*(\B)$. Since it is faithful, $C^*(\B)$  can be identified with $\hat{\pi}_u(C^*(\B))$. We call the weak closure $W^*(\B) = \hat{\pi}_u(C^*(\B))''$ the \emph{universal von Neumann algebra} of $\B$, in other words, the enveloping von Neumann algebra of $C^*(\B)$ \cite[Ch.12]{Dix82}. Accordingly, every representation $\pi$ of $C^*(\B)$ extends to a unique normal representation of $W^*(\B)$ and similarly, every endomorphisms $\rho$ of $C^*(\B)$ extends to a unique normal endomorphism of 
$W^*(\B)$. Throughout this paper, when no confusion arises we will again denote by $\pi$ and $\rho$ these normal extensions. 
\end{definition}

\begin{remark}\label{rem:CFT-univC}
\begin{itemize}
\item[(1)]  If $\hat{\pi}$ is the representation of $C^*(\B)$ corresponding to the representation $\pi$ of $\B$ according to the above universal property, we shall freely say that $\hat{\pi}$ is locally normal, localized in some $I_0\in\I$, the vacuum representation, or covariant, respectively, if $\pi$ is such. The statistical dimension $d(\hat{\pi})$ of $\hat{\pi}$ is defined by  $d(\hat{\pi}) =d(\pi)$.
Moreover, we shall drop the $\hat{\cdot}$ sign when no confusion arises. Note that the terms intertwiner, 
unitary equivalence, irreducibility and direct sum for representations of the net $\B$ agree with the standard terminology for the corresponding representations  of the C*-algebra $C^*(\B)$. 
 
\item[(2)] $C^*(\B)$ inherits the local structure from the net $\B$. Thus, when no confusion arises, we may identify $\B(I)$ with its image
 $\iota_I(\B(I))$ in $C^*(\B)$.
\item[(3)] We say that an endomorphism $\rho$ of $C^*(\B)$ is \emph{localized in $I_0$} if it is the identity endomorphism in restriction to the subalgebra $\iota_{I}(\B(I))$ whenever $\overline{I} \subset I_0'$. It is well known (see Proposition \ref{prop:CFT-endoms} below) that there is a natural one-to-one correspondence between localized covariant endomorphisms of $ C^*(\B)$ and the localized locally normal representations of $C^*(\B)$, which in turn correspond to the localized representations of $\B$ by definition. 

We then write
\begin{equation}\label{eq:CFT-Delta}
\begin{aligned}
\Delta^0:=& \bigcup_{I\in\I} \Delta^0_I,\\
\Delta^0_I :=& \{ \PSL-\textrm{covariant endomorphisms of $C^*(\B)$ localized in } I\}.
\end{aligned}
\end{equation}

\item[(4)] As a consequence of the universal property of $C^*(\B)$ there is a unique representation $\alpha$ of $\Diff$ by automorphisms of $C^*(\B)$ implementing the covariance of $\B$, \ie such that $\alpha_g(\iota_I(x)) = \iota_{gI}(U(g)xU(g)^*)$. 
It follows from  \cite{DFK} that $\alpha_g$ is an inner automorphism of $C^*(\B)$, for all $g\in \Diff$, and we shall use the same notation for its lift to $\Diff^{(\infty)}$.
\end{itemize}
\end{remark}

Let $\mathcal U$ be an open neighborhood of the identity in $\PSI$.  A map $z$ from $\mathcal U$ into the set of unitary operators of 
$C^*(\B)$ is said to be  a local (unitary) $\alpha$-cocycle if  $z(gh) = z(g) \alpha_{g}(z(h))$ whenever $g, h, gh \in {\mathcal U}$. 
If $\mathcal U=\PSI$ then $z$ is said to be an $\alpha$-cocycle. From the fact that $\PSI$ is simply connected it follows that every local 
$\alpha$-cocycle defined on a connected open neighborhood of the identity in $\PSI$ has a unique extension to an $\alpha$-cocycle, cf. \cite[Sect.8]{GL1}.  

Now let $\pi$ be a locally normal representation of $C^*(\B)$ localized in an interval $I_0 \in \I$. Fix an interval $I\in \I$ containing the closure of $I_0$ and define the open neighborhood of the identity ${\mathcal U}_{I_0,I}$ to be the connected component of the identity of the open set
 $$\{g \in \PSI: \dot{g}\overline{I_0} \subset I \}.$$ 
Then, $U_\pi(g) U(g)^* \in \B(I')'= \B(I)$ for all $g \in {\mathcal U}_{I_0, I}$ and the map $z^I_\pi :  {\mathcal U}_{I_0,I} \to C^*(\B)$ defined by 
\begin{equation}\label{eq:CFT-cocycle-def}
z^I_\pi(g) := \iota_{I}(U_\pi(g) U(g)^*),\quad  g\in {\mathcal U}_{I_0,I},
\end{equation}
is a local unitary $\alpha$-cocycle and therefore extends to a unique $\alpha$-cocycle  which can be easily seen to be independent on the choice of $I$ and which will be denoted by $z_\pi$, cf. \cite[Sect.8]{GL1}. Note that if $I$ is any interval in $\I$ containing the closure of $I_0$ then we have $z_\pi(g) = \iota_{I}(U_\pi(g) U(g)^*)$ for all $g \in {\mathcal U}_{I_0, I}$. As a consequence, 
$\pi_0(z_\pi(g)) = U_\pi(g) U(g)^*$ for all $g \in {\mathcal U}_{I_0, I}$ and hence, since $ {\mathcal U}_{I_0, I}$ is a neighborhood of the identity,
$\pi_0(z_\pi(g)) = U_\pi(g) U(g)^*$ for all $g \in \PSI$.

\begin{proposition}\label{prop:CFT-endoms}
There is a natural one-to-one correspondence between $\PSL$-cova\-riant representations $\pi$ of $C^*(\B)$ localized in a given interval $I_0\in\I$ and $\PSL$-covariant endomorphisms $\rho$ of $C^*(\B)$ localized in $I_0$: given the representation $\pi$, $\rho$ is the unique endomorphism of $C^*(\B)$ localized in $I_0$ satisfying $\pi=\pi_0\circ\rho$ and the covariance condition 
\[
\Ad (z_\pi(g)^*)\circ \rho = \alpha_{g}\circ\rho\circ \alpha_{g}^{-1}, \quad g\in\PSI.
\]
If $\rho_1$ and $\rho_2$ are localized covariant endomorphisms of $C^*(\B)$ then $\pi_0\circ \rho_1$ is equivalent to $\pi_0 \circ \rho_2$ 
iff $\rho_1 = \Ad(u)\circ \rho_2$ for some unitary $u\in C^*(\B)$. A localized covariant endomorphism $\rho$ of $C^*(\B)$ is C*-algebra automorphism if and only if its statistical dimension $d(\rho):= d(\pi_0 \circ \rho)$ is equal to one.  
\end{proposition}

\begin{proof}
The proof is mainly given in \cite[Sect.8(p.541)]{GL1}, see also \cite[Sect.5.1]{FRS2}. $\rho$ is uniquely determined by the condition $\rho|_{\B(I)}:= \Ad z_\pi(g)$, 
$g\in\PSI$ , $\overline{I} \subset \dot{g}I_0'$ and the universal property of $C^*(\B)$. Recall that, because of \eqref{eq:CFT-cocycle-def}, $z_\pi(g)$ is uniquely determined by $\pi$ through the unique inner representation $U_\pi$ of $\PSI$ making $\pi$ covariant.

We give now a proof for the last statement. If the covariant endomorphism $\rho$ of $C^*(\B)$ localized in the interval  $I_0$ is an automorphism, then it is straightforward to see that $\rho^{-1}$ is again a covariant endomorphism localized in $I_0$. Therefore, 
$1=d(\id)=d(\rho \rho^{-1}) = d(\rho)d(\rho^{-1})$ so that $d(\rho)=1$. Vice versa, if $d(\rho)=1$, then it follows from 
\cite[Cor.2.10]{GL} that $\pi_0\circ \rho$ is irreducible. By \cite[Thm.2.11]{GL} there exists a covariant endomorphism $\bar{\rho}$ 
localized in $I_0$ (the conjugate endomorphism) such that $\pi_0\circ \bar{\rho}$ is irreducible,
$\pi_0\circ \bar{\rho}\rho \simeq \pi_0\circ \rho\bar{\rho}$ contains $\pi_0$ as a subrepresentation and $d(\bar{\rho}) =d(\rho) = 1$. 
Since $d(\bar{\rho}\rho)=d(\bar{\rho})d(\rho)=1$, also $\pi_0\circ \bar{\rho}\rho$ is irreducible and hence unitarily equivalent to $\pi_0$. Therefore $\bar{\rho}\rho$ is an inner automorphism of $C^*(\B)$. Similarly $\rho\bar{\rho}$ is an inner automorphism. Hence $\rho$ is an automorphism. 
\end{proof}

We shall say that two localized endomorphisms $\rho_1, \rho_2$ of $C^*(\B)$ are equivalent if \linebreak 
$[\pi_0 \circ \rho_1]=[\pi_0 \circ \rho_2]$ and write $[\rho]$ for the equivalence class of the localized covariant endomorphism $\rho$ of $C^*(\B)$. 

We shall need the following proposition, cf. \cite{Lo97,Lo01}.  

\begin{proposition}\label{prop:CFT-endoms2}
Given $I_0\in\I$, we define the unitary-valued map
\[
 z: (\rho,g) \in \Delta^0_{I_0}\times \PSI \mapsto z_{\pi_0 \circ \rho}(g).
\]
Then
\begin{equation}\label{twovarcocycle1}
z(\rho,gh) = z(\rho,g) \alpha_{g}(z(\rho,h)) \quad \rho \in\Delta^0_{I_0}, \; g,h\in \PSI.
\end{equation}
Moreover, if $\pi_0 \circ \rho\sigma(C^*(\B))' \cap \pi_0 \circ \rho(C^*(\B))''$ is a direct sum of finite dimensional 
algebras then 
\begin{equation}\label{twovarcocycle2}
z(\rho\sigma,g) =  \rho(z(\sigma,g)) z(\rho,g), \quad \rho,\sigma\in\Delta^0_{I_0}, \; g\in \PSI.
\end{equation}
In particular, \eqref{twovarcocycle2} always holds whenever $\rho$ and $\sigma$ have finite statistical dimension,
namely the restriction of $z$ to the endomorphisms with finite statistical dimension is a two-variable cocycle.
\end{proposition}

\begin{proof}
The first of the two identities is true by definition. Let us therefore prove the second one. 

Note that the map 
$g\mapsto \rho(z(\sigma,g)) z(\rho,g)$ is a unitary $\alpha$-cocycle so that it is enough to show that it coincides with 
$g\mapsto z(\rho\sigma,g)$ in a neighborhood of the identity in $\PSI$. 
Let us consider the unitary representation $\tilde{U}$
of $\PSI$ on the vacuum Hilbert space $\H$ of $\B$ defined by 
$\tilde{U}(g):= \pi_0\left(\rho(z(\sigma,g)) z(\rho,g)\right) U(g)$, $g\in \PSI.$
It satisfies $\tilde{U}(g)\pi_0\circ\rho\sigma(x)\tilde{U}(g)^*=\pi_0\circ\rho\sigma(\alpha_g(x))$ for all $x\in C^*(\B)$ and all $g \in \PSI$. 
Accordingly, $V(g) :=U_{\pi_0\circ\rho\sigma}(g)\tilde{U}(g)^* \in \pi_0\circ\rho\sigma(C^*(\B))'$ for all $g\in \PSI$. It follows that the map $g\to V(g)$ defines a strongly continuous unitary representation $V$ of $\PSI$ on $\H$ with values in $\pi_0\circ\rho\sigma(C^*(\B))'$. 

On the other hand 
$\tilde{U}(g)= \pi_0(\rho(z(\sigma,g))) U_{\pi_0\circ \rho}(g) \in \pi_0\circ \rho(C^*(\B))''$ so that 
$V(g) \in \pi_0\circ\rho\sigma(C^*(\B))' \cap \pi_0\circ \rho(C^*(\B))''$ for all $g\in \PSI$. Since by assumption 
$\pi_0\circ\rho\sigma(C^*(\B))' \cap \pi_0\circ \rho(C^*(\B))''$ is a direct sum of finite dimensional algebras, $V$ is a direct sum of finite dimensional unitary representations, and hence it must be trivial because $\PSI$ has no nontrivial finite dimensional unitary representations. 
Accordingly $\pi_0\left(\rho(z(\sigma,g)) z(\rho,g)\right)= \pi_0(z(\rho\sigma,g))$ for all $g\in \PSI$. Hence, for any $I\in \I$ containing the closure of $I_0$, we have $\rho(z(\sigma,g)) z(\rho,g)= z(\rho\sigma,g)$ for all $g\in {\mathcal U}_{I_0,I}$, because, if $g\in{\mathcal U}_{I_0,I}$ then $\rho(z(\sigma,g)), z(\rho,g), z(\rho\sigma,g) \in \iota_I(\B(I))$ and the restriction of $\pi_0$ to the latter subalgebra is faithful. Therefore, the two $\alpha$-cocycles coincide in a neighborhood of the identity and the conclusion follows. 
\end{proof}

Clearly, the localized covariant endomorphism corresponding to the vacuum representation $\pi_0$ of $\B$ is the identity automorphism 
$\id : C^*(\B) \to C^*(\B)$. The unitary $z_\pi(g)$ turns out to transport the ``charges" related to $\pi$ from $I_0$ to $\dot{g} I_0$, wherefore it is also called a \emph{charge transporter} of the representation $\pi$ between the localization regions $I_0$ and $\dot{g}I_0$.

In every locally normal irreducible representation $\pi$ of $\B$ the conformal Hamiltonian $L_0^\pi$ is selfadjoint with a lowest eigenvalue $\lw(\pi)=h_\pi$,  called the \emph{lowest energy} of $\pi$, and discrete spectrum equal to $\lw(\pi)+\NN$ (except for the vacuum representation case where possibly the spectrum is strictly contained in  $\NN$). 

Let us now return to our graded-local conformal net $\A$. The concept of representations begun in Definition \ref{def:CFT-reps1} becomes slightly more involved here:

\begin{definition}[{\cite[Sect.4]{CKL}}]
\begin{itemize}
\item[$(1)$] A \emph{$G$-covariant soliton} of $\A$ is a family $\pi=(\pi_I)_{I\in\bar{\I}_\R}$ of normal representations of $\A$ restricted to $\bar{\I}_\R$ on a common Hilbert space $\H_\pi$ with a projective unitary representation $U_\pi:G^{(\infty)}\ra B(\H_\pi)$ such that:
\begin{itemize}
\item[-] $\pi_{I_2}|_{\A(I_1)}=\pi_{I_1}$, whenever $I_1\subset I_2$, $I_1,I_2\in\bar{\I}_\R$;
\item[-] for every $I\in\I_\R$, 
\[
U_\pi(g) \pi_I(x) U_\pi(g)^*=\pi_{\dot{g}I}(U(g)xU(g)^*),\quad g\in \mathcal{V}_I , x\in\A(I), 
\]
where $\mathcal{V}_I$ is the connected component of the identity in $G^{(\infty)}$ of the open set \linebreak
$\{g \in G^{(\infty)}: \dot{g}I \in \I_\R \}$. 
\end{itemize}

If $U_\pi$ is a positive energy representation, namely the selfadjoint generator $L^\pi_0$
corresponding to the one parameter group of rotations has nonnegative spectrum, we say that $\pi$ has positive energy.
If $\pi$ is a $G$-covariant soliton and the family $ \pi = (\pi_I)_{I\in\bar{\I}_\R}$ can be extended 
to $\I$ giving a covariant DHR representation of $\A$ (with the same $U_\pi$) we say that $\pi$ is a \emph{DHR representation} of $\A$.

\item[$(2)$] A \emph{$G$-covariant general soliton} $\A$ is a $G$-covariant soliton such that the restriction of $\pi$ from $\A$ to the even subnet $\A^\gamma$ is a DHR representation.  In case $G=\Diff$, we shall simply say \emph{general soliton}.
\item[$(3)$] A $G$-covariant general soliton $\pi$ of $\A$ is called \emph{graded} if there exists a selfadjoint unitary $\Gamma_\pi\in B(\H)$ commuting with $U_\pi(g)$, for all $g\in G^{(\infty)}$, and such that
\[
\Gamma_\pi \pi_I(x) \Gamma_\pi = \pi_I(\gamma(x)) , \quad x\in\A(I),I\in\bar{\I}_\R.
\]
\item[$(4)$] A $G$-covariant graded general soliton $\pi$ of $\A$ is called \emph{supersymmetric} if $L_0^\pi-\lambda \unit$ admits an odd square-root (the {\it supercharge}) for some $\lambda \in \R$.
\end{itemize}
\end{definition}

\begin{remark}\label{rem:CFT-IR}
It can be shown (using a straight-forward reasoning based on covariance relations) that a family $(\pi_I)_{I\in\I_\R}$ of normal representations of $\A$ which is covariant with respect to a given projective unitary representation of $G^{(\infty)}$ extends automatically from $\I_\R$ to $\bar{\I}_\R$, thus defines a $G$-covariant soliton. We shall make use of this (simplifying) fact when considering the super-current algebra models and the super-Virasoro net in the final section.
\end{remark}

As in the case of DHR representations (see the comments after Definition \ref{def:CFT-reps1}) it can be shown that in various cases, as a consequence of the results in \cite{Wei06}, the positive energy condition is automatic for $G$-covariant general solitons, see \cite[Prop.12\&Prop.21]{CKL}. In particular an irreducible $G$-covariant general soliton is always of positive energy. 

\begin{example}[\textbf{Super-Virasoro net}]\label{ex:superVir}
The fundamental example of a graded-local conformal net is the \emph{super-Virasoro net}. It has been introduced in \cite[Sect.6]{CKL} and studied in \cite{CHKL} with the aim of constructing spectral triples. We sketch here the main ideas.

The \emph{Neveu-Schwarz super-Virasoro algebra} is the $\Z/2$-graded Lie algebra generated by even $L_n$, $n\in\N$, odd $G_r$, $r\in\frac12+\Z$, and a central even element $\hat{c}$, together with the following (anti-) commutation relations
\begin{equation}
\begin{gathered}\label{eq:superVir}
    [L_m , L_n] = (m-n)L_{m+n} + \frac{\hat{c}}{12}(m^3 - m)\de_{m+n, 0},\\
    [L_m, G_r] = \Big(\frac{m}{2} - r\Big)G_{m+r},\\
    [G_r, G_s]_+ = 2L_{r+s} + \frac{\hat{c}}{3}\Big(r^2 - \frac14\Big)\de_{r+s,0}.
    \end{gathered}
\end{equation}
The \emph{Ramond super-Virasoro algebra} is defined analogously but with $r\in\Z$. Both are equipped with an involution: $L_n\mapsto L_{-n}$, $G_{r}\mapsto G_{-r}$ and $\hat{c}\mapsto \hat{c}$.

In an irreducible (nonzero) unitary positive energy representation $\pi$ of these algebras, the central element $\hat{c}$ is represented by a positive multiple $c\unit$ of $\unit$; $\pi$ is completely determined by this number $c$ (the {\it central charge}) together with the lowest eigenvalue $h_\pi= \lw(\pi)\in\R_+$ of $L_0^\pi$. More generally, if $\pi$ is a possibly reducible positive energy unitary representation such that $\hat{c}$ is represented by a positive multiple $c\unit$ we say that $\pi$ has central charge $c$.
The vacuum representation with central charge $c$ of the Neveu-Schwarz algebra is the graded irreducible representation $\pi$ with central charge $c$ where the lowest eigenvalue of $L_0^\pi$ is $0$. The Ramond vacuum representation with central charge $c$ is the graded irreducible representation of the Ramond algebra having central charge $c$ and lowest energy $h_\pi=c/24$. For representations of the Ramond algebra we automatically have a (odd) square-root of the conformal Hamiltonian up to an additive constant $L_0^\pi- \frac{c}{24} \unit$, namely $G_0^\pi$. We shall consider this point in more generality in Proposition \ref{prop:CFT-supersymmetric}.

Consider now the Neveu-Schwarz algebra in the vacuum representation with certain central charge value $c$, and drop the symbol $\pi$ for simplicity. For smooth and localized functions $f\in\Cci(\S)_I$ with $I\in\I$ (or $ I \in\I_\R$), the Fourier coefficients $f_n$, $n\in\Z$, (or $f_r:=\frac{1}{2\pi}\int_{-\pi}^\pi f(\rme^{\rmi t})\rme^{-\rmi r t}\rmd t$, $r\in\frac12+\Z$, respectively) are rapidly decreasing and, owing to so-called energy bounds (analytical properties of the operators $L_n,G_r$ on $\H$), the formal sums
\[
 \sum_{n\in\Z} f_n L_n,\quad \sum_{r\in\frac12+\Z} f_r G_r
\]
are densely defined closable essentially selfadjoint operators on $\Cci(L_0)$, which forms an invariant core for them; we denote their selfadjoint closures, the so-called smeared fields, by $L(f)$ and $G(f)$, respectively. Similarly one can define smeared fields 
$L^\pi(f)$ and $G^\pi(f)$ for any unitary positive energy representation of the Neveu-Schwarz or Ramond super-Virasoro algebra having a given central charge $c$, \cf \cite{CHKL}.
Since
\[
[L(f),L(g)]=[L(f),G(g)]=[G(f),G(g)]_+ = 0, \quad \supp(f) \cap \sup (g) =\emptyset,
\]
one can show (\cf \cite[Sect. 6]{CKL}) that the family of von Neumann algebras
\[
 \A_{\SVir,c} (I) := \{ \rme^{\rmi L(f)}, \rme^{\rmi G(f)}:\; f\in\Cci(\S)_I\}'', \quad I\in\I_\R,
\]
extends (by covariance) to a unique graded-local conformal net $(\A_{\SVir,c}(I))_{I\in\I}$ over $\S$. 
\end{example}

Based on the fact that every local conformal net contains the Virasoro net as a minimal conformal subnet \cite[Prop.3.5]{KL04} and \cite{Carpi}[Rem. 3.8] (in fact irreducible by \cite[Prop.3.7]{Carpi}), one can make the following definition: 

\begin{definition}\label{def:CFT-SCFTnet}
The net $\A$ with central charge $c$ is \emph{superconformal} if it contains $\A_{\SVir,c}$ as a $\PSL$-covariant subnet and the projective representation $U$ of $\Diff^{(\infty)}$ making $\A$ diffeomorphism covariant satisfies
\[
U(\Diff^{(\infty)}(I)) \subset \A_{\SVirc}(I) \subset \A(I), \quad I\in\I.
\]
\end{definition}

If $\A$ is superconformal, then it can be shown, using \cite[Prop.3.7]{Carpi}, that the inclusion $\A_{\SVirc}(I) \subset \A(I)$ is irreducible, for every $I\in\I$, in other words, $\A_{\SVirc} \subset \A$ is an irreducible inclusion of conformal nets. Conversely, if $\A$ contains $\A_{\SVirc}$ as a $\PSL$-covariant irreducible subnet and $c<3/2$, then $\A$ is superconformal, \cf \cite[Sect.7]{CKL}.

Let us focus a little bit more on the several kinds of solitons, their properties, and the meaning of supersymmetry.

\begin{proposition}[{\cite[Sect.4.3]{CKL}}]\label{prop:CKL22}
Let $\pi$ be an irreducible general soliton of $\A$. Then the following three conditions are equivalent:
\begin{itemize}
 \item[-] $\pi$ is graded,
 \item[-] $\pi |_{\A ^\gamma}$ is reducible,
 \item[-] $\pi |_{\A ^\gamma} \simeq  \pi_+ \oplus \pi_+ \circ \hat{\gamma} =: \pi_+\oplus \pi_-$,
\end{itemize}
with $\pi_+$ and $\pi_-$ inequivalent irreducible localized DHR representation of $\A^\gamma$ and $\hat{\gamma}$ a localized DHR automorphism of $\A^\gamma$ dual to the grading.
\end{proposition}

It can be shown \cite[Cor.23 (proof)]{CKL} that for irreducible graded $\pi$ and under the assumption of finite statistical dimension on $\pi_+$, $\rme^{\rmi 4\pi L_0^\pi}= \rme^{\rmi 4\pi h_\pi}\unit$ is a scalar and we have in fact the two possibilities
\[
 \rme^{\rmi 2\pi L_0^\pi} =  \rme^{\rmi 2\pi L_0^{\pi_+}}\oplus \pm \rme^{\rmi 2\pi L_0^{\pi_+}}
= \rme^{\rmi 2\pi h_\pi}\unit \oplus \pm \rme^{\rmi 2\pi h_\pi}\unit,
\]
while in the irreducible ungraded case $\rme^{\rmi 2\pi L_0^\pi}$ is always a scalar because it commutes with $\bigvee_{I\in\I} \pi_I(\A^\gamma(I))=B(\H_\pi)$. Here ``$+$" will correspond to $(R)$ in the following theorem, ``$-$" to $(NS)$.  In the following, we shall not restrict ourselves to finite index, but we will always assume that $\rme^{\rmi 4\pi L_0^\pi}$ is a scalar if $\pi$ is irreducible, although this assumption might turn out to be unnecessary. As we shall see, it is easy to verify this for the models in Section \ref{sec:loop}.

\begin{theorem}[{\cite[Sect.4.3]{CKL}}]\label{th:CFT-RS-N}
Let $\pi$ be an irreducible general soliton of $\A$  such that $\rme^{\rmi 4\pi L_0^\pi}$ is a scalar, and denote $\pi|_{\A^\gamma}=:\pi_+ \oplus \pi_+ \circ \hat{\gamma}$ or $\pi|_{\A^\gamma}=:\pi_+$ with an irreducible representation $\pi_+$ of $\A^\gamma$ (for graded or ungraded $\pi$, respectively). Then $\pi$ is of either of the subsequent two types:
\begin{itemize}
\item[$(NS)$] $\pi$ is a DHR representation of $\A$; equivalently,
$\rme^{\rmi 2\pi (L_0^\pi-h_\pi \unit)}=\Gamma_\pi$ implements the grading.
\item[$(R)$]   $\pi$ is not a representation but only a general soliton of $\A$; equivalently $\rme^{\rmi 2\pi (L_0^\pi-h_\pi \unit)}= \unit$ and hence does not implement the grading.
\end{itemize}
In case $(NS)$, $\pi$ is called a \emph{Neveu-Schwarz representation} of $\A$, and in case $(R)$, a \emph{Ramond representation}, the latter, however, being actually only a general soliton and not a proper (DHR) representation of $\A$. A direct sum of irreducible Neveu-Schwarz (Ramond) representations is again called a Neveu-Schwarz (Ramond) representation. 
\end{theorem}

\begin{proposition}\label{prop:CFT-supersymmetric}
Assume $\A$ is superconformal with central charge $c$ and let $\pi$ be an irreducible $G$-covariant general soliton of $\A$ such that $\rme^{\rmi 4\pi L_0^\pi}$ is a scalar. 
\begin{itemize}
\item[-] If $\pi$ is supersymmetric, then it is a Ramond representation. 
\item[-] Conversely, if $\pi$ is a Ramond representation, then there is a unitary representation with central charge $c$ of the Ramond algebra by operators $G_r^\pi,L_n^\pi$ on $\H_\pi$ such that $\pi_I(\rme^{\rmi G(f)})=\rme^{\rmi G^\pi(f)}$ and $\pi_I(\rme^{\rmi L(f)})=\rme^{\rmi L^\pi(f)}$, for all $f\in\Cci(\S)_I$ and $I\in\I_\R$, and the choice $Q=G^\pi_0$ makes $\pi$ supersymmetric.
\end{itemize}
\end{proposition}

\begin{proof}
Suppose $\pi$ is a Neveu-Schwarz representation of $\A$. Then it has to be graded by definition: $\Gamma_\pi=\rme^{\rmi2\pi (L_0^\pi -h_\pi)}$, with $h_\pi$ the lowest energy of $\pi$. Suppose there exists an odd supercharge $Q$ for that representation. Then 
$Q^2=L_0^\pi -\lambda \unit$ for some $\lambda \in \R$. According to the preceding theorem, the grading is implemented by $\Gamma_\pi = \rme^{\rmi 2\pi (L_0^\pi -h_\pi)}$, so
\[
 Q = -\Gamma_\pi Q \Gamma_\pi^* = -\rme^{\rmi 2\pi (L_0^\pi-h_\pi)} Q \rme^{-\rmi 2\pi (L_0^\pi-h_\pi)} = -Q
\]
because $Q$ commutes with $L_0^\pi=Q^2+ \lambda \unit$. This is a clear contradiction, so $\pi$ cannot be supersymmetric.

Suppose now instead that $\pi$ is a (graded or ungraded) Ramond representation. The net $\A$ contains the super-Virasoro net $\A_{\SVir,c}$ as an irreducible conformal subnet, with $c$ the central charge of $\A$, and the representation $\pi$ restricts to a (possibly reducible) \emph{Ramond} representation $\hat{\pi}$ of this subnet on $\H_{\hat{\pi}}=\H_\pi$ since $\rme^{\rmi 2\pi L_0^{\hat{\pi}}}=\rme^{\rmi 2\pi L_0^\pi}$ is a scalar. The image of the net $\A_{\SVir,c}$ under a Ramond representation $\hat{\pi}$ is isomorphic to the net $\A_{\SVir,c}^{\hat{\pi}}$ generated directly by the smeared super-Virasoro fields in the  corresponding Ramond representation $\hat{\pi}$ of the Lie algebra $\SVir$ on $\H_{\hat{\pi}}$, as shown in \cite{CW}. But in such a representation of the super-Virasoro algebra, a possible supercharge on $\H_{\pi}$ is $G_0^{\hat{\pi}}$, which follows easily from \eqref{eq:superVir}. Since $(G_{0}^{\hat{\pi}})^2 = L_0^\pi - \frac{c}{24}\unit$, $Q:=G_{0}^{\hat{\pi}}$ forms in fact a supercharge for $\A$ in the representation $\pi$. Moreover, if $\pi$ is graded, $Q$ is odd.
\end{proof}

\section{Noncommutative geometry}\label{sec:NCG}

In noncommutative geometry the notion of spectral triples (called K-cycles in \cite[IV.2.$\gamma$]{Co1}) is fundamental. Depending on the context there are several ways of defining it. For background information and versions adapted to the setting of superconformal field theory consider \cite[Sect.3]{CHKL}. In this paper we content ourselves with the most common

\begin{definition}\label{def:NC-STs}
A {\em $\theta$-summable spectral triple} $(A , (\pi,\H), Q)$ consists of
\begin{itemize}
\item[-] a $*$-algebra $A$
\item[-] a separable Hilbert space $\H$ and a $*$-representation (not necessarily faithful) $\pi:A\ra B(\H)$;
\item[-] a selfadjoint operator $Q$ on $\H$ such that $\rme^{-tQ^2}$ is trace-class, for all $t>0$, and such that $\pi(A)\subset \dom(\delta)$, with $\delta$ the derivation on $B(\H)$ induced by $Q$ (\cf Eq. \eqref{eq:NC-domain1} below).
\end{itemize}
The spectral triple is called \emph{even} if there is a grading $\Gamma$ on $\H$ such that $[\Gamma,\pi(A)]=0$ and $\Gamma Q\Gamma = -Q$. Otherwise, it is called \emph{odd}.
\end{definition}

If $\H$ is a Hilbert space and $\Gamma$ is a grading operator on $\H$, we have the decomposition $\H =\H_+ \oplus \H_-$ of $\H$ as a direct sum of the corresponding even and odd subspaces. Now, if $T$ is a (possibly unbounded and densely defined) operator on $\H$ which is odd, i.e., such that $\Gamma T \Gamma = -T$, 
then we can write 
$$T = \left( 
\begin{array}{cc}
0 & T_-  \\
T_+ & 0 
\end{array} \right), $$ 
with operators $T_\pm$ from (a dense subspace of) $\H_\pm$ to $\H_\mp$.
 Accordingly, if $(A , (\pi,\H), Q)$ is a $\theta$-summable spectral triple then for the selfadjoint $Q$
we can write 
$$Q = \left( 
\begin{array}{cc}
0 & Q_-  \\
Q_+ & 0 
\end{array} \right), $$ 
with $Q_-= Q_+^*$. On the other hand, if $T$ is even, i.e., it commutes with $\Gamma$, then we can write 
$$T = \left( 
\begin{array}{cc}
T_+ & 0  \\
0 & T_- 
\end{array} \right), $$ 
with operators $T_\pm$ from (a dense subspace of) $\H_\pm$ to $\H_\pm$

We recall now that if $Q$ is a selfadjoint operator on $\H$ (not necessarily graded), then one obtains a derivation $\delta$ as follows: $\dom(\delta)$ is the set of elements $x\in B(\H)$ such that
\begin{equation}\label{eq:NC-domain1}
(\exists y\in B(\H))\quad x Q \subset Q x -y,
\end{equation}
in which case $\delta(x) := y$. If $\H$ is graded and $Q$ odd, then one also obtains a superderivation $\delta_s$ in a similar manner: $\dom(\delta_s)$ is the set of elements $x\in B(\H)$ such that
\begin{equation}\label{eq:NC-domain2}
(\exists y\in B(\H))\quad \Gamma x \Gamma Q \subset Q x -y, 
\end{equation}
in which case $\delta_s(x) := y$; clearly, the restrictions of $\delta_s$ and $\delta$ to the even elements are derivations and coincide: $\delta_s |_{\dom(\delta_s)^\Gamma}=\delta |_{\dom(\delta)^\Gamma}$. In either of the two cases, $\dom(\delta)$ (or $\dom(\delta_s)$) equipped with the norm $\|\cdot\|+\|\delta(\cdot)\|$ (or $\|\cdot\|+\|\delta_s(\cdot)\|$, respectively) becomes a Banach $*$-algebra \cite[Cor.2.3]{CHKL}. We remark that based on such a superderivation we introduced the concept of graded spectral triples in \cite{CHKL}, which, however, shall play no role in the present article.

A remarkable generalization of de Rham homology of currents in differential geometry to the noncommutative setting was found by Connes to be cyclic cohomology.  Here we will consider its entire version which is suitable for applications to quantum field theory. For extensive discussions and further references consider the standard textbooks \cite{Co1,GBVF}.

\begin{definition}\label{def:NC-JLO}
(1) Let $(A,(\|\cdot\|_i)_{i\in I})$ be a locally convex unital $*$-algebra and, for any nonnegative integer $n$, let $C^n(A)$ be the vector space of multilinear maps : $A\times (A/{\C \unit})^n \to \C$. We will identify $C^n(A)$ with the space of $(n+1)$-linear forms $\phi$ on $A$ such that $\phi(a_0,a_1, ...,a_n)=0$ if $a_i= \unit$ for some $1\leq i \leq n$ (simplicial normalization). For integers $n < 0$ we set $C^n(A) := \{0\}$. Let $C^\bullet(A) := \prod_{k=0}^\infty C^k(A)$ be the space of sequences $\phi= (\phi_k)_{k \in \N_0}$, $\phi_k \in C^k(A)$ and define the operators $b:C^\bullet(A) \to C^\bullet(A)$ and $B:C^\bullet(A) \to C^\bullet(A)$ by 

\begin{align*}
(b\phi)_k (a_0, ..., a_k) :=& \sum_{j=0}^{k-1} (-1)^j \phi_{k-1}(a_0, ..., a_j a_{j+1}, ..., a_k)\\
& + (-1)^k \phi_{k-1}(a_k a_0, a_1,  ..., a_{k-1}),\\
(B\phi)_k(a_0, ..., a_k)  := & \sum_{j=0}^k (-1)^{jk} \phi_{k+1}(\unit, a_j, ..., a_n, a_0, ..., a_{j-1}).
\end{align*}

The linear map $\partial: C^\bullet(A) \ra C^\bullet (A)$ defined by $\partial := b + B$ satisfies $\partial^2=0$ and, with 
 the boundary operator $\partial$, $C^\bullet(A)$ becomes the \emph{cyclic cocomplex} $C^\bullet(A) = (C^e(A),C^o(A))$ over $\Z/2\Z$, namely the elements of  $C^e(A)= \prod_{k=0}^\infty C^{2k}(A)$ (the {\it even cochains})
are mapped into the elements of $C^o(A)= \prod_{k =0}^\infty C^{2k+1}(A)$ (the {\it odd cochains}) and vice versa.

(2) A cochain $\phi= (\phi_k)_{k\in\NN} \in C^\bullet(A)$ is called \emph{entire} if, for every bounded subset $B\subset A$, there is a constant $c_B$ such that
\[
 |\phi_k(a_0,...,a_k)| \le \frac{1}{\sqrt{k!}} c_B, \quad a_i\in B, k\in\NN.
\]
Letting $CE^\bullet(A)$ be the entire elements in $C^\bullet(A)$, one defines the \emph{entire cyclic cohomology} $(HE^e(A), HE^o(A))$ of $A$ as the cohomology of the cocomplex $((CE^e(A),CE^o(A)), {\partial})$. The cohomology class of an entire {\it cocycle} 
$\phi \in CE^\bullet(A) \cap \ker (\partial)$ will be denoted by $[\phi]$.
\end{definition}

Concerning entireness, there are a few alternative conventions in the original literature \cite{Co1,GS,JLO1}  but what matters is actually only that everything is chosen in a consistent way. Usually the setting is that of Banach algebras while here we are dealing only with locally convex algebras, a generalization discussed in \cite[IV.7.$\alpha$]{Co1}, \cf also \cite{Mey}. The above entireness condition for a cyclic cochain $\phi$ can be reformulated as follows: for every bounded subset $B\subset A$ and every $\lambda>0$, there is $c_{B,\lambda}$ such that
\[
 |\phi_k(a_0,...,a_k)| \le \frac{1}{\sqrt{k!}} c_{B,\lambda} \lambda^k, \quad a_i\in B,k\in\N,
\]
or again in another way: for every bounded subset $B\subset A$, we have
\[
 \limsup_{k\ra\infty} \left( \sqrt{k!} \sup_{a_i\in B} |\phi_k(a_0,...,a_k)| \right)^{1/k} = 0.
\]
In the case of a Banach algebra $A$, it suffices to study the unit sphere as bounded subset, and there we obtain the classical entireness condition $\limsup_{k\ra\infty} (\sqrt{k!}\|\phi_k\|)^{1/k} = 0$ from \cite{GS,JLO1}.

We now recall, (cf. \eg \cite{Bl, Co1,Cor11}) the definition of the K-groups for a unital locally convex algebra $(A,\|\cdot\|_{i\in I})$:
\begin{itemize}
\item[$(K_0)$]  Let $\Mat_r(A)$ be the  locally convex algebra of $r\times r$ matrices over $A$, $r\in \N$, and let $\Mat_\infty(A)$ denote the 
algebra of infinite matrices over $A$ with only finitely many nonzero entries. The maps $x\mapsto \diag(x,0)$ define natural embeddings $\Mat_r(A) \ra \Mat_{r+1}(A)\ra ...\ra \Mat_\infty(A)$. We denote by $P\Mat_r(A)$ the set of idempotents in $\Mat_r(A)$, $r\in \N \cup \{\infty\}$. An equivalence relation on $P\Mat_\infty(A)$ is defined by $p \sim q$ if there are $x, y \in \Mat_\infty(A)$ such that $p=xy$, $q=yx$. There is a binary operation 
$$(p_1,p_2)\in P\Mat_{r_1}(A) \times P\Mat_{r_2}(A)  \mapsto  p_1\oplus p_2 := \diag(p_1, p_2) \in P\Mat_{r_1+r_2}(A),$$ 
which turns $P\Mat_\infty(A)/\sim$ into an abelian semigroup. Then the \emph{$K_0$-group of $A$} is defined as
\[
  K_0(A) := \textrm{Grothendieck group of}\; P\Mat_\infty(A)/\sim,
\]
where the Grothendieck group of an arbitrary additive semigroup $H$ is the group of formal differences of elements of 
$H$, i.e.,
\[
  (H\times H) /\{(h_1,h_2) \sim_{H\times H} (g_1,g_2)\; \Leftrightarrow \; (\exists k\in H) h_1+g_2 +k = h_2 + g_1 +k \}.
\]
We write $[p]$ for the element in $K_0(A)$ induced by $p\in P\Mat_\infty (A)$.

\item[$(K_1)$] Let $\GL_r(A)$ denote the group of invertible elements in $\Mat_r(A)$.  With the diagonal inclusion $u\in  \GL_r(A) \mapsto u\oplus 1:= \diag (u,1) \in  \GL_{r+1}(A)$, this gives a directed family, and the inductive limit $\GL_\infty(A) := \varinjlim \GL_r(A)$, with the inductive limit topology,  is a topological group. Its connected component of the identity is denoted by $\GL_\infty(A)_0$. Then the \emph{$K_1$-group of $A$} is defined as the quotient
\[
 K_1(A) := \GL_\infty(A)/ \GL_\infty (A)_0
\]
and it turns out to be abelian. 
We write $[u]$ for the element in $K_1(A)$ induced by $u \in \GL_\infty(A)$.    
\end{itemize}

In order to describe the pairing between entire cyclic cohomology and K-theory we need to introduce a canonical extension of a linear functional 
$\phi_k \in C^k(A)$  to a linear functional $\phi_k^r \in C^k(\Mat_r(A))$, defined as follows after the 
identification $\Mat_r(A) \simeq \Mat_r(\C) \otimes A$,  

\begin{equation}
 \phi_k^r (m_0 \otimes a_0, .... ,m_k\otimes a_k) :=  \tr (m_0...m_k)  \phi_k(a_0, .... ,a_k).
\end{equation}
The map $\phi \mapsto \phi^r$ is a morphism of the complexes of entire chains, see \cite[IV.7.$\delta$]{Co1}.
Let $(A , (\pi,\H), Q)$ be a $\theta$-summable spectral triple. We denote by $\pi_r$ the representation of $\Mat_r(A)$ on 
$\H_r:= \C^r \otimes \H$ defined by $\pi_r(m \otimes a) = m \otimes \pi(a)$, $m \in \Mat_r(\C)$, $a\in A$.  Moreover, for every operator $T$ on $\H$ we 
consider the operator $T_r = \unit \otimes T$ on $\H_r$. Then, for every $r\in \N$,  $(\Mat_r(A) , (\pi_r,\H_r), Q_r)$ is a $\theta$-summable spectral triple which is even with grading $\Gamma_r$ if $(A , (\pi,\H), Q)$ is  even with grading $\Gamma$.

The concluding main theorem about $\theta$-summable spectral triples and entire cyclic cohomology is (in chronological order) mainly due to \cite{Connes88,JLO1,GS,JLO2,CP1} in the case of Banach algebras. The basic ingredients are the JLO cochains  associated to a spectral triple $(A , (\pi,\H), Q)$. They are obtained from the $(n+1)$-linear forms $\tau_n$ on $A$ defined by
\begin{equation}\label{eq:JLO-def}
\begin{aligned}
 \tau_{n}(a_0,...a_n) =
\int_{0\le t_1\le ...\le t_n \le 1}\tr \Big( & {\Gamma} {\pi}(a_0) \rme^{-t_1 {Q}^2} [{Q}, {\pi}(a_1)]\rme^{-(t_2-t_1) {Q}^2} ... \\
& ...[{Q},{\pi}(a_n)]\rme^{-(1-t_n) {Q}^2} \Big) \rmd t_1 ... \rmd t_n ,
\end{aligned}
\end{equation}
for $n>0$ and $\tau_0(a_0)= \tr(\Gamma \pi(a_0) \rme^{-Q^2})$,
where we take ${\Gamma}= \unit$ if the spectral triple is odd. 
The locally convex version we shall need can be found basically in \cite[IV.7]{Co1} and reads as follows:
\begin{theorem}\label{th:JLO}
Let $A$ be a unital locally convex $*$-algebra, and let $(A,(\pi,\H),Q)$
be a $\theta$-summable spectral triple such that the representation $\pi$ of $A$ in the Banach algebra $\dom (\delta)$ is continuous.
\begin{itemize}
\item[$(1)$] If the spectral triple is even (odd), then the cochain $(\tau_{n})_{n\in 2\NN}$ ($(\tau_{n})_{n\in 2\NN+1}$) is an even (odd) entire cyclic cocycle, called the {\em even (odd) JLO cocycle} or \emph{Chern character}.
\item[$(2)$] Suppose $(Q_{t})_{t\in[0,1]}$ is a differentiable homotopy between the two (odd) self-adjoint operators $Q_{0},Q_{1}$, \ie the domain of $Q_t$ does not depend on $t\in [0,1]$ and $t\mapsto \overline{Q_{t}-Q_{0}}$ is a norm differentiable $B(\H)$-valued function. Then for the corresponding even or odd JLO cocycles $\tau^t$ we have $[\tau^s] = [\tau^t]$, for $s,t\in[0,1]$, namely the entire cohomology class of $\tau^t$ does not depend on $t\in [0,1]$. \index{norm-differentiable homotopy}
\item[$(3)$] The values of the maps  $(\phi,p) \in \left( CE^e(A) \cap \ker (\partial)  \right) \times P\Mat_r(A) \mapsto \phi (p) \in \C$, 
$r \in \N$, where 

\[ 
\phi (p) := \phi_0^{r}(p) + \sum_{k=1}^\infty (-1)^k \frac{(2k)!}{k!} \phi_{2k}^{r}((p-\frac12),p,...,p) , 
\]
only depend on the cohomology class of the entire cocycle $\phi$ and on the K-theory class of the idempotent $p$ and are additive on the latter. Hence they give rise to a pairing $\langle [\phi], [p] \rangle := \phi(p)$ between the even entire cyclic cohomology $HE^e(A)$ and K-theory $K_0(A)$. Moreover, the operator
$\pi_r(p)_{-} {Q_r}_{+} \pi_r(p)_+$ from $\pi_r(p)_+{\H_r}_+$ to $\pi_r(p)_-{\H_r}_-$ is a Fredholm  operator and for the even JLO cocycle 
$\tau$ we have
\[
\tau (p) = \langle [\tau], [p] \rangle = \ind\big( \pi_r(p)_{-} {Q_r}_{+} \pi_r(p)_+ \big) \in \Z .
\]

The values of the maps  $(\phi,u) \in \left( CE^o(A) \cap \ker (\partial)  \right) \times \GL_r(A) \mapsto \phi (u) \in \C$, 
$r \in \N$, where 
\[
\phi (u) := \frac{1}{\sqrt{\pi}}\sum_{k=0}^\infty (-1)^k k! \phi^r_{(2k+1)}(u^{-1}, u, ..., u^{-1},u),
\]

only depend on the cohomology class of the entire cocycle $\phi$ and on the K-theory class of $u$ and are additive on the latter. Hence they give rise to a pairing $\langle [\phi], [u] \rangle := \phi(u)$ between odd entire cyclic cohomology $HE^o(A)$ and K-theory $K_1(A)$. Moreover, if $\pi_r(u)$ is unitary, the operator $\chi_{[0,\infty)}(Q_r) \pi_r (u) \chi_{[0,\infty)}(Q_r)$ from 
$\chi_{[0,\infty)}(Q_r)\H_r$ to $\chi_{[0,\infty)}(Q_r)\H_r$ is a Fredholm operator  and for the odd JLO cocycle  $\tau$ we have
\[
\tau(u) =  \langle [\tau], [u] \rangle =   
\ind \big( \chi_{[0,\infty)}(Q_r) \pi_r (u) \chi_{[0,\infty)}(Q_r) \big) \in \Z. 
\]

\end{itemize}
\end{theorem}

One usually encounters this theorem in the context of Banach algebras. However, it extends to the setting of locally convex algebras. Concerning the generalization of points (1) and (2), we just have to check that the locally convex entireness conditions are satisfied by $\tau$, and then follow the lines of \eg \cite{CP1,GS}: the continuity of $\pi : A \to \dom(\delta)$ implies its boundedness, \ie bounded sets in $A$ are mapped into bounded sets in the Banach algebra $\dom(\delta)$, so that the entireness of the JLO cochain associated to $(\dom(\delta), ({\rm id} , \H), Q)$ implies the entireness of the JLO cochain for $(A,(\pi,\H),Q)$. The pairing with K-theory in part (3) is proved in \cite[IV.7. $\delta$, Theorem 21] {Co1} (even case) and  
\cite[IV.7.$\epsilon$, Corollary 27]{Co1}, cf. also \cite{Getz93}, (odd case). The index formula in the even case follows from \cite[IV.8.$\delta$, Theorem 19]{Co1} and 
\cite[IV.8.$\epsilon$, Theorem 22]{Co1}. The index formula in the odd case for Banach $*$-algebras follows from \cite[Corollary 7.9]{CP1} and 
\cite[Theorem 10.8]{CP1}. Hence the formula is true for the $\theta$-summable spectral triple $(\dom(\delta), ({\rm id} , \H), Q)$. Accordingly the formula for $(A,(\pi,\H),Q)$ with $A$ locally convex follows from the continuity of $\pi :  A \to \dom(\delta)$ and the formula for the pairing with $K_1$ for general locally convex algebras.  

\medskip

We shall need the following proposition in Section \ref{sec:gen-ST}. Consider the JLO cocycle $\tau$ associated to a spectral triple 
$(A,(\pi,\H),Q)$, even (with grading operator $\Gamma$) or odd. Let $v \in \dom(\delta)$ be a unitary operator 
commuting with $\Gamma$ if the spectral triple is even and let $\pi^v$ be the continuous representation of $A$ on $\H$ defined by
$\pi^v(a):=v\pi(a)v^*$, $a\in A$. Then $(A,(\pi^v,\H),Q)$ is again a spectral triple (with grading operator $\Gamma$ in the even case) and we denote by $\tau^v$ the corresponding JLO cocycle.

\begin{proposition}\label{prop:perturb-v}
$[\tau^v] = [\tau]$ for every unitary $v \in \dom(\delta)$ (commuting with $\Gamma$ in the even case). 
\end{proposition}

\begin{proof} First note that the JLO cocycle $\tau^v$ coincides with the JLO cocycle associated to the spectral triple 
$(A,(\pi,\H), v^*Qv)= (A,(\pi,\H), Q +v^*\delta(v))$. Now, the family $Q_t:=Q + tv^*\delta(v)$ with $t\in [0,1]$ is a differentiable homotopy between $Q$ and $v^*Qv$, so that $[\tau] = [\tau^v]$ according to Theorem \ref{th:JLO} (2). 
\end{proof}

\section{Spectral triples and cyclic cocycles for superconformal nets}\label{sec:gen-ST}

We briefly recall our objective from the Introduction. Given a superconformal net $\A$ with grading automorphism $\gamma$ and acting on the (vacuum) Hilbert space $\H$, we would like to associate in a canonical way a (family of) $\theta$-summable spectral triple(s), which in a second step should give rise to entire cyclic cocycles corresponding to certain localized endomorphisms of $C^*(\A^\gamma)$. These cocycles will be investigated in the index pairing of the next section.

\subsection*{The spectral triples}

In order to construct our spectral triples, we need a suitable representation $(\pi,\H_\pi)$ of $\A$ and want $\A$ to have a supercharge for the conformal Hamiltonian: an odd selfadjoint operator $Q$ on $\H_\pi$ such that $Q^2= L_0^\pi$ up to an additive constant. The existence of $Q$ depends on the representation $\pi$ of $\A$, and according to Proposition \ref{prop:CFT-supersymmetric} it exists precisely if $\pi$ is a Ramond representation, in which case our fixed choice shall be $Q=G_0^\pi$. Let $(\pi_R,\H_R)$ denote henceforth a certain fixed (graded or ungraded) irreducible Ramond representation of $\A$, to be described in more detail below. By restriction it gives rise to a (possibly reducible) representation $\pR |_{\A^\gamma}$ of $\A^\gamma$ and thus to a representation of $C^*(\A^\gamma)$ and to 
its unique normal extension to $W^*(\A^\gamma)$, for which we write simply $\pR$ again whenever confusion with the original representation of $\A$ is unlikely. Note then that, for a localized endomorphism $\rho$ of $C^*(\A^\gamma)$, $\pR\circ\rho$ is again a representation of $C^*(\A^\gamma)$ or $W^*(\A^\gamma)$ on the same Hilbert space $\H_R$. Throughout the rest of the paper we shall make the 

\bigskip
\noindent\textbf{Standing Assumption.}  $\A$ satisfies the \emph{split property} as in Definition \ref{def:CFT-split} and the \emph{trace-class condition} in the representation $\pR$, namely
\begin{equation}\label{eq:gen2-trace-class}
\tr_{\H_R} \Big(\rme^{-t L_0^{\pR}}\Big) < \infty,\quad t>0.
\end{equation}

Now a few observations. First, Proposition \ref{prop:CKL22} tells us how to distinguish between graded and ungraded representations, corresponding to the even and odd index pairing in Theorem \ref{th:JLO}(3). Second, in view of Theorem \ref{th:JLO} a K-theoretical index pairing makes sense only for the \emph{even} subnet $\A^\gamma$: physically $G_0^\pR$ being an odd element in the super-Virasoro algebra should induce a superderivation while in the index pairing we need the induced derivation, so we have to deal with $\A^\gamma$ where they coincide. Third, we would like our spectral triples to exhibit some aspects of the sector structure of this subnet $\A^\gamma$. Working with the \emph{local} algebras $\A^\gamma(I)$ lets us face a serious obstruction as we shall see later in Proposition \ref{prop:gen2-tau1}. Finally, if possible we would like to establish a correspondence between equivalence of representations and equivalence of the corresponding JLO entire cyclic cocycles. Based on $\A$, $\pR$ and $Q$, we therefore have to construct a couple $(\Delta,\AA_\Delta)$ consisting of a suitable family $\Delta\subset \Delta^0$ of localized endomorphisms and a subalgebra $\AA_\Delta$ of the (global) universal von Neumann algebra 
$W^*(\A^\gamma)$ equipped with a topology such that the JLO cocycles associated to the family $(\AA_\Delta,(\pR \circ \rho,\H_R),Q)_{\rho\in\Delta}$ of spectral triples become entire. After that we can study the question of equivalence.

Let us recall from \eqref{eq:CFT-cocycle-def} and Proposition \ref{prop:CFT-endoms2} that, for every endomorphism $\rho$ of 
$C^*(\A^\gamma)$ localized in $I$, one can define the cocycle $z(\rho,g) \in C^*(\A^\gamma)$, $g\in\PSI$. 
It satisfies
\begin{equation}\label{eqcovcocycle1}
\rho^g:= \alpha_g\circ\rho\circ \alpha_g^{-1}=\Ad(z(\rho,g)^*)\circ \rho, \quad g\in \PSI.
\end{equation}
Moreover if the closure of $I$ is contained in some $I_0 \in \I$ and ${\mathcal U}_{I,I_0}$ is the connected component of the identity of the 
open set $\{g\in \PSI : \dot{g}\overline{I} \subset I_0\}$ we have 
$z(\rho,g) = \iota_{I_0}(U_{\pi^\gamma_0 \circ \rho}(g) U^\gamma(g)^* )$ for all $g \in {\mathcal U}_{I,I_0}$. 

Notice also that  it follows from \eqref{eqcovcocycle1} and Proposition \ref{prop:CFT-endoms2} that
\begin{equation}\label{eqcovcocycle2}
z(\rho^h,g)= z(\rho,h)^*z(\rho,g)\alpha_g(z(\rho,h)) = z(\rho,h)^*z(\rho,gh)
,\quad  g, h \in \PSI.
\end{equation}

\begin{definition}\label{def:gen2-Delta}
A covariant endomorphism $\rho\in\Delta^0$ of $C^*(\A^\gamma)$ is called \emph{differentiably transportable} if it is localized in some $I\in\I$ and $\pi_R (z(\rho,g)) \in \dom(\delta)$, for all $g\in\PSI$.
The set of differentiably transportable endomorphisms localized in $I$ is denoted by $\Delta_I^1$, and
we set $\Delta^1 := \bigcup_{I\in\I} \Delta_I^1$. 
\end{definition}

\begin{remark}\label{remarkDelta^1cov} 
As a consequence of \eqref{eqcovcocycle2} we have that, for any $g\in \PSI$, $\rho \in \Delta_I^1$ if and only if $\rho^g\in \Delta_{gI}^1$.
\end{remark}

Now, for any $I \in \I$, let $p_{I}$ be the middle point of $I$ and let $P_I$ be the dilation-translation subgroup of $\PSI$ fixing $p_{I'}$. 

\begin{proposition}\label{prop:gen2-Delta}
Let $\rho\in\Delta^0_I$, $I\in \I$ and let ${\mathcal U}$ be an open subset of $\PSI$ containing $P_I$.
Assume that 
$\pi_R (z(\rho,g)) \in \dom(\delta)$ for all $g\in {\mathcal U}$. Then $\rho\in\Delta^1_I$. 
In particular, an endomorphism $\rho\in\Delta^0_I$ belongs to $\Delta^1_I$ if and only if $\pi_R (z(\rho,g)) \in \dom(\delta)$ for all 
$g\in {\mathcal U}_I$, where ${\mathcal U}_I$ is the connected component of $\unit$ of the open set 
$\{g\in \PSI: p_{I'} \notin \dot{g}\bar{I} \}$. 
\end{proposition}

\begin{proof} 
Let $\rho$ be a covariant endomorphism localized in the fixed interval $I$ and let ${\mathcal U}$ be an open subset of $\PSI$ containing 
$P_I$. Assume that $\pi_R (z(\rho,g)) \in \dom(\delta)$ for all $g\in {\mathcal U}$. Let $g$ be an arbitrary element of $\PSI$ and let 
$r(t)$, $t\in \R$, be (the lift to $\PSI$ of) the one-parameter subgroup of rotations. It easily follows from the Iwasawa decomposition of 
$\operatorname{SL}(2,\R)$ (see e.g. \cite[Appendix I]{FG}) that $g=r(s)p$ for some $s\in \R$ and $p\in P_I$. Accordingly 
 $z(\rho,g)=z(\rho,r(s) p) = z(\rho,r(s)) \alpha_{r(s)}(z(\rho,p))$ 
 and hence 
  $\pR(z(\rho,g))= \pR(z(\rho,r(s))) \rme^{\rmi sL_0^\pR}\pR((z(\rho,p)))\rme^{-\rmi sL_0^\pR}.$
 Now, $\pR((z(\rho,p))) \in \dom(\delta)$ by assumption and $\rme^{\rmi tL_0^\pR}$ commutes with $Q$ for every $t\in \R$.  Hence
 $\rme^{\rmi sL_0^\pR}\pR((z(\rho,p)))\rme^{-\rmi sL_0^\pR}$ belongs to $\dom(\delta)$.
 
On the other hand if $n$ is a sufficiently large positive integer then $r({s}/{n}) \in {\mathcal U}$ so that $\pR(z(\rho,r({s}/{n}))) \in \dom(\delta)$. Accordingly  
\begin{align*} 
 \pR(z(\rho,r(s)))=&\pR\Big( z(\rho,r(s/n)) \alpha_{r(s/n)}(z(\rho,r(s/n))) \dots \alpha_{r(s-s/n)}(z(\rho,r(s/n)))\Big)\\
=&\pR\big(z(\rho,r({s}/{n}))\big)\rme^{\rmi \frac{s}{n} L_0^\pR}\pR\big(z(\rho,r({s}/{n}))\big) \dots  \pR\big(z(\rho,r({s}/{n}))\big)\rme^{-\rmi(s- \frac{s}{n})L_0^\pR}
\end{align*}
lies in $\dom(\delta)$. Therefore $ \pR(z(\rho,g)) \in \dom(\delta)$ and the claim follows. 
\end{proof} 

The above proposition will be very useful in order to check that localized endomorphisms are differentiably transportable. 
This is because for $\rho$ localized in $I$ and $g \in {\mathcal U}_I$ there is an open interval $I_1 \subset \S \setminus \{ p_{I'}\}$ containing $\bar{I}\cup \dot{g}\bar{I}$ 
such that the explicit formula $z(\rho,g)=\iota_{I_1}(\pi^\gamma_0(z(\rho,g))) = \iota_{I_1}\left(U_{\pi^\gamma_0\circ \rho}(g)
U^\gamma(g)^* \right)$ holds, so that many computations become easier or possible at all.

\begin{proposition}\label{prop:indpR} 
Let $I \in \I_\R$, let $x\in \A^\gamma(I)$ and let  $\phi_{I}$ be a real smooth function with support in $\S\setminus \{-1\}$ and coinciding with $1$ on $I$. Then, $\pR(x) \in \dom(\delta)$ if and only if $x$ is in the domain of the derivation 
 $[G(\phi_{I}),\cdot ]$ and, in this case, $\delta(\pR(x))= \pR([G(\phi_{I}), x ])$. As a consequence,  for every $I \in \I$, the algebra
  $\{x \in \A^\gamma(I): \pR(x) \in \dom(\delta) \}$ does not depend on the choice of the Ramond representation $\pR$. Moreover, $\Delta^1_I$ does not depend on the choice of $\pR$.
 \end{proposition}
 
 \begin{proof} Let $I \in \I_\R$, $x\in \A^\gamma(I)$ and let  $\phi_{I}$ be a real smooth function with support in $I_0\subset\S\setminus \{-1\}$ and coinciding with $1$ on $I$. According to \cite[Sect.5]{CHKL} on localized implementations of the canonical superderivation and together with  Proposition \ref{prop:CFT-supersymmetric} and the local normality of $\pi_R$ we have
\begin{align*}
 [G^\pR_0,\pR(x))] =& [G^{\pR}(\phi_{I}),\pR(x)]
= \frac{\rmd}{\rmi\rmd t} \Ad(\rme^{\rmi t G^{\pR}(\phi_{I})})(\pR(x)) \big|_{t=0}\\
=&\frac{\rmd}{\rmi\rmd t} \pR\big(\Ad(\rme^{\rmi t G(\phi_{I})})(x)\big) \big|_{t=0} \\
=& \pR \circ \iota_{I_0} ([G(\phi_{I}), x]).
\end{align*}
As a consequence, for every $I \in \I_\R$, the algebra   $\{x \in \A^\gamma(I): \pR(x) \in \dom(\delta) \}$ does not depend on the choice of 
$\pR$ and the same is true for an arbitrary $I \in \I$ as a consequence of covariance and the fact that $G^\pR_0$ is invariant under rotations. 

Now let $\rho$ be a covariant endomorphism of $C^*(\A^\gamma)$ localized in an interval $I\in \I$. Then, by Proposition \ref{prop:gen2-Delta}, 
$\rho \in \Delta^1_I$ if and only if $\pR(z(\rho,g)) \in \dom (\delta)$ for all $g\in \mathcal{U}_I$. Now, given $g \in \mathcal{U}_I$, there is an interval $I_1 \in \I$ such that $\pi_0^\gamma(z(\rho,g)) \in \A^\gamma(I_1)$. Accordingly  $\pR(z(\rho,g)) \in \dom (\delta)$ if and only if 
$\pi_0^\gamma(z(\rho,g)) \in \{x \in \A^\gamma(I_1): \pR(x) \in \dom(\delta) \}$, a condition that does not depend on the choice of $\pR$. 
It follows that $\Delta^1_I$ does not depend on the choice of $\pR$. \end{proof}

\begin{remark}\label{remark:gradedpR}
In view of Proposition \ref{prop:CKL22}, in later applications where a graded $\pR$ is given and $\hat{\gamma}$ is an explicit fixed localized endomorphism representing the dual of the grading, we may assume (after replacing $\pR$ with a unitarily equivalent representation) that $\pR |_{C^*(\A^\gamma)}= \pi_{R,+}\oplus \pi_{R,+}\circ \hat{\gamma}$. As well known, without loss of generality, we may also assume that 
$\hat{\gamma}^2 = \id$.
\end{remark} 

We keep on record the following consequence of the cocycle identity:
\begin{equation}\label{eq:gen2-alpha-z}
 \pi_R(\alpha_h(z(\rho,g))) = \pi_R(z(\rho,h)^*)\pR(z(\rho,hg))\in\dom(\delta),\quad h,g\in\PSI.
\end{equation}

Recalling from \cite[Sect.2\&A]{Lo97} the concept of tensor C*-categories we can exhibit the following structure of $\Delta_I^1$:

\begin{proposition}\label{prop:gen2-tenscat}
The subset of $\Delta_I^1$ consisting of differentiably transportable endomorphisms with finite statistical dimension is closed under composition and conjugates. Moreover, the braiding operators $\eps(\rho,\sigma)$ lie in $\pR^{-1}(\dom(\delta))$. In particular $\Delta^1_I$ with morphisms the corresponding intertwiners in $\A^\gamma(I)$ forms a braided tensor C*-category with conjugates and simple unit.
\end{proposition}

\begin{proof}
\emph{Composition}. Given $\rho,\sigma\in\Delta^1_I$, choose $h\in\PSI$ such that $\dot{h}\bar{I}\cap \bar{I}=\emptyset$. Then $\rho^h$ acts trivially on $\A(I)$, so owing to Proposition \ref{prop:CFT-endoms2} we have
\[
 z(\rho\sigma,g)=  \rho(z(\sigma,g)) z(\rho,g) 
= \Ad(z(\rho,h))\circ \rho^h (z(\sigma,g)) z(\rho,g) =\Ad(z(\rho,h))(z(\sigma,g))  z(\rho,g)  ,
\]
which is a product of elements in $\pR^{-1}(\dom(\delta))$. 

\emph{Conjugation}. Suppose that $\rho$ is localized in $I$. Let $I_1 \in \I$ be the interval with boundary points $p_I$ and $p_{I'}$ (the middle points of $I$ and $I'$ respectively) in anti-clockwise order so that $r_{I_1}$ is the $\S$-reflection fixing 
$p_I$ and $p_{I'}$.  Choose an arbitrary $g$ in the open set ${\mathcal U}_I$ defined in Proposition 
\ref{prop:gen2-Delta}. Then $z(\rho,g)\in\A^\gamma(I_0)$ with a certain open interval $I_0 \subset \S \setminus \{ p_{I'}\}$ 
containing $\bar{I}\cup \dot{g}\bar{I}$. Since $r_{I_1}$ fixes $p_{I'}$ we can also assume that $I_0 =r_{I_1} I_0$ so that 
$\bar{I}\cup \dot{g}\bar{I}\cup r_{I_1}\dot{g}\bar{I} \subset I_0$. 

Let $\alpha_{r_{I_1}}$ be the antilinear automorphism of $C^*(\A^\gamma)$ implementing the reflection $r_{I_1}$ (it exists and is unique because of the universal property of $C^*(\A^\gamma)$).  We recall from \cite[Sect.8]{GL1}  and \cite[Sect.2]{GL} the formula 
$\bar{\rho} := \alpha_{r_{I_1}} \circ \rho\circ \alpha_{r_{I_1}}$ for an explicit choice of representative of the conjugate sector. By the Bisognano-Wichmann property (2) in Section \ref{sec:CFT}, the anti-unitary $U(r_{I_1})$ representing the reflection $r_{I_1}$ on $\H$ equals $\hat{J}_{I_1}:= ZJ_{I_1} $, where $J_{I_1}$ is the modular conjugation of the local algebra $\A(I_1)$ with respect to the vacuum vector $\Omega$. Similarly, the anti-unitary $U^\gamma(r_{I_1})$ representing the reflection $r_{I_1}$ on $\H^\Gamma$ coincides with the modular conjugation $J^\gamma_{I_1}$, of the local algebra $\A^\gamma(I_1)$ on $\H^\Gamma$ with respect to the vacuum vector $\Omega$. Note that $J^\gamma_{I_1}=J_{I_1}|_{\H^\Gamma}=\hat{J}_{I_1}|_{\H^\Gamma}$.
Then a straight-forward computation based on the covariance of $\rho$ and $\bar{\rho}$ shows that 
$U_{\pi_0^\gamma \circ \bar{\rho}}(g)= J^\gamma_{I_1} U_{\pi_0^\gamma \circ \rho}(g^r) J^\gamma_{I_1}$, with 
$g^r:=r_{I_1} g r_{I_1}$, so

\begin{align*}
z(\bar{\rho},g) &= \iota_{I_0}\left(U_{\pi_0^\gamma \circ \bar{\rho}}(g)U^\gamma(g)^*\right) 
= \iota_{I_0}\left(J^\gamma_{I_1}U_{\pi_0^\gamma \circ \rho}(g^r)J^\gamma_{I_1} U^\gamma(g)^*\right)\\
&=  \iota_{I_0}\left(J^\gamma_{I_1}U_{\pi_0^\gamma \circ \rho}(g^r)U^\gamma(g^r)^*J^\gamma_{I_1}  \right)
= \alpha_{r_{I_1}}(z(\rho,g^r)) \in \A^\gamma(r_{I_1} I_0) = \A^\gamma(I_0).
\end{align*}

By Proposition \ref{prop:indpR}, the derivation $\delta = [G_0^\pR,\cdot]$ restricted to $\pR(\A(I_0))$ comes from a derivation in the vacuum representation, given by the commutator with $G(\phi_{I_0})$, where $\phi_{I_0}$ is any smooth function compactly supported in some proper open interval of $\S$ and satisfying $\phi_{I_0}|_{I_0}=1$. We choose this function symmetric so that $\phi_{I_0}\circ r_{I_1}= \phi_{I_0}$. Then $\hat{J}_{I_1}G(\phi_{I_0})\hat{J}_{I_1}^* = - G(\phi_{I_0}\circ r_{I_1}) = - G(\phi_{I_0})$ which follows from an adaptation of \cite[Sect.3]{BSM} to fermionic fields. On the even subnet, we clearly have
 $\Ad(\hat{J}_{I_1})|_{\A^\gamma}=\Ad(J_{I_1})|_{\A^\gamma}$. Hence, for every $\psi_1,\psi_2\in\Cci(L_0)$,
\begin{align*}
 \langle \psi_1,G(\phi_{I_0}) z(\bar{\rho},g)\psi_2\rangle &= \langle \psi_1,G(\phi_{I_0}) \hat{J}_{I_1}z(\rho,g^r) \hat{J}_{I_1}^*\psi_2\rangle\\
&= - \langle \psi_1,\hat{J}_{I_1} G(\phi_{I_0})  z(\rho,g^r) \hat{J}_{I_1}^* \psi_2\rangle\\
&=  - \langle \psi_1,\hat{J}_{I_1} z(\rho,g^r) G(\phi_{I_0})\hat{J}_{I_1}^* \psi_2\rangle 
- \langle \psi_1,\hat{J}_{I_1} [G(\phi_{I_0}),z(\rho,g^r)] \hat{J}_{I_1}^* \psi_2\rangle \\
&=  \langle \psi_1,z(\bar{\rho},g)  G(\phi_{I_0}) \psi_2\rangle  
- \langle \psi_1,\hat{J}_{I_1} [G(\phi_{I_0}),z(\rho,g^r)] \hat{J}_{I_1}^* \psi_2\rangle.
\end{align*}
Using again Proposition \ref{prop:indpR}, we thus obtain $\pi_R(z(\bar{\rho},g)) \in\dom(\delta)$ with
\[
\delta(\pR(z(\bar{\rho},g)) )=\pR([G(\phi_{I_0}) z(\bar{\rho},g)]) 
= - \pR(\hat{J}_{I_1} [G(\phi_{I_0}),z(\rho,g^r)] \hat{J}_{I_1}^*).
\]

\emph{Braiding.} Let $\rho,\sigma\in\Delta_I^1$ and let $I_0 \in \I$ be an interval containing the closure of $I$. 
Consider $g_\pm\in\PSI$ such that $\dot{g}_\pm I$ is localized on the left (right) of $I$ inside $I_0$ and $\dot{g}_\pm\bar{I}\cap \bar{I}=\emptyset$, respectively. Then $z(\rho,g_\pm)\in\A^\gamma(I_0)$, and the braiding operator is given by \cite[(2.2)]{Reh}
\[
\eps^\pm(\sigma,\rho) = \Ad(z(\rho,g_\pm)^* )\sigma(z(\rho,g_\pm)) \in \A^\gamma(I_0).
\]
Since $I_0 \supset \overline{I}$ was arbitrary, we have $\eps^\pm(\sigma,\rho) \in \A^\gamma(I)$ by outer regularity (the same result can be obtained by using Haag duality and the fact that $\eps^\pm(\sigma,\rho)$ intertwines $\sigma\rho$ and $\rho\sigma$). 
Now let $h \in \PSI$ be such that $\dot{h}I \cap I = \emptyset$. Then 
\[
\sigma(z(\rho,g_\pm)) = \Ad(z(\rho,h))\circ \sigma^h (z(\rho,g_\pm))=\Ad(z(\rho,h))(z(\rho,g_\pm)).
\]
Hence, since $\rho,\sigma\in\Delta^1$,  $\pR(\eps(\rho,\sigma)^\pm)$ is a product of operators in $\dom(\delta)$, so $\pR(\eps(\rho,\sigma)^\pm)\in\dom(\delta)$. 
\end{proof}

We can sharpen the statement about conjugates in the following special situation:

\begin{proposition}\label{prop:gen2-invert}
If $\rho\in\Delta_I^1$ is an automorphism, then $\rho^{-1}\in\Delta_I^1$, \ie every $\Delta^1_I$ is closed under inverses if they exist as endomorphisms of $C^*(\A^\gamma)$.
\end{proposition}

In other words, apart from the possible choice of representative $\Ad(U(r_I))\circ \rho\circ \Ad(U(r_I))$ of the conjugate of $\rho$, one may also choose $\rho^{-1}$.

\begin{proof}
For $g\in {\mathcal U}_I$, we have $z(\rho,g)\in \A^\gamma(I_0)$ with suitable $I_0\in\I$. Recall that locally, and in particular on $\A^\gamma(I_0)$, $\rho$ is implemented by a local unitary $z(\rho,h)$, with a suitable $h\in\PSI$ depending on $I_0$. Hence, for any such $g$ and corresponding $h$, we have
\[
 \unit = z(\id,g) = z(\rho,g)\rho(z(\rho^{-1},g)) = z(\rho,g)\Ad(z(\rho,h))(z(\rho^{-1},g)),
\]
which implies $z(\rho^{-1},g)= \Ad(z(\rho,h)^*)(z(\rho,g)^*) \in \pi_R^{-1}(\dom(\delta))$, as a product of elements in $\pi_R^{-1}(\dom(\delta))$. Applying then Proposition \ref{prop:gen2-Delta}, we get $\rho^{-1}\in\Delta_I^1$.
\end{proof}

Let us turn now to the definition of the differentiable algebra, based on a given unital subset $\Delta\subset\Delta^0$, not necessarily in $\Delta^1$, and recall that we use the same symbol for the normal extension to $W^*(\A^\gamma)$ of an endomorphism or a representation of $C^*(\A^\gamma)$.

\begin{definition}\label{def:gen2-AA}
Given a superconformal net $\A$ with a (supersymmetric) Ramond representation $(\pi_R,\H_R)$ and supercharge $Q$, and given a subset $\Delta\subset \Delta^0$ of localized endomorphisms of $C^*(\A^\gamma)$ which contains the identity automorphism $\id$, the associated \emph{$\Delta$-$\delta$-differentiable algebra} is the $*$-algebra given by
\[
 \AA_{\Delta} := \{ x\in W^*(\A^\gamma) :\; (\forall \rho\in\Delta)\; \pi_R\circ\rho (x) \in \dom(\delta) \}
\]
and the corresponding local subalgebras are given by $\AA_{\Delta}(I) := \AA_{\Delta} \cap \A^\gamma(I)$. Endowed with the family of norms
\[
 \|\cdot\|_\rho := \|\cdot\|_{W^*(\A^\gamma)} + \|\delta(\pi_R\circ\rho(\cdot))\|_{B(\H_R)},\quad \rho\in\Delta,
\]
 $\AA_{\Delta}$ becomes a locally convex $*$-algebra. 
\end{definition}

\begin{remark}\label{rem:gen2-AA}
(1) If $\Delta$ actually forms a semi-group, \ie $\rho\sigma\in \Delta$ for all $\rho,\sigma\in\Delta$, then $\AA_\Delta$ is globally \emph{invariant} under $\Delta$: for $\sigma\in\Delta$ and $x\in \AA_\Delta$, we have $\pR\circ\rho (\sigma(x)) = \pR\circ(\rho\sigma)(x) \in \dom(\delta)$ for all $\rho \in \Delta$, so $\sigma(\AA_\Delta)\subset \AA_\Delta$ and $\sigma:\AA_\Delta \to \AA_\Delta$ is continuous (with respect to the locally convex topology on $\AA_\Delta$). 

(2) Since we are only interested in locally normal representations of $\A^\gamma$, we could consider the corresponding locally normal C*-algebra $C^*_{\operatorname{ln}}(\A^\gamma)$ instead of $C^*(\A^\gamma)$, which is characterized by the corresponding universal property in Definition \ref{def:CFT-univC} for locally normal representation only. It was explicitly  defined in \cite{CCHW} and proven to be $\sigma$-weakly closed 
in the universal locally normal representation if $\A^\gamma$ is completely rational. As a consequence, in the latter case, $\AA_\Delta$ may alternatively be chosen as a subalgebra of $C^*_{\operatorname{ln}}(\A^\gamma)$. Although a wide range of models including a relevant part of those studied in Section \ref{sec:loop} are known to fall into this class, $W^*(\A^\gamma)$ does not cause any particular difficulties and we do not want to become too restrictive wherefore we continue here with the general setting above.
\end{remark}

\begin{theorem}\label{theorem:gen2-AAST}
Given $\A$, $\pR$ and $\Delta$ as above, $(\AA_\Delta,(\pR\circ \rho,\H_R),Q)$ is a $\theta$-summable spectral triple, for every $\rho\in\Delta$. It is even if $\pR$ is graded and odd if $\pR$ is ungraded. Moreover, the representation $\pR\circ \rho$ of the locally convex algebra $\AA_\Delta$ into the Banach algebra $\dom(\delta)$ is continuous.
\end{theorem}

\begin{proof}
Note simply that the trace-class property in \eqref{eq:gen2-trace-class} guarantees the $\theta$-summability. By definition we have 
$\pR \circ \rho(\AA_\Delta)\subset \dom(\delta)$. Finally, in the case of graded $\pR$, we have $\pR\circ \rho(\AA_\Delta)\subset 
\pR\circ \rho(W^*(\A^\gamma)) \subset \pR(W^*(\A^\gamma))\subset B(\H_R)^\Gamma$ and $Q=G_0^\pR$ is odd, so the spectral triple is even. The continuity of 
$\pR \circ \rho :\AA_\Delta \to \dom(\delta)$ follows easily from the definition of the locally convex topology on $\AA_\Delta$. 
\end{proof}

We shall collect a few further important properties of $\AA_\Delta$. For the proof we need

\begin{lemma}\label{lem:gen2-locAA}
Let $I_1, I_2\in\I$ with $\bar{I}_1\subset I_2$. Then there is a unitary $V\in B(\H\otimes\H,\H)$ such that $V(a_1\otimes a_2') V^* = a_1 a_2'$
for all $a_1 \in \A(I_1)$ an all $a_2' \in Z\A(I_2')Z^*$, and $V(\Gamma\otimes \Gamma)V^*=\Gamma$. So $V$ implements an isomorphism $\A(I_1)\vee Z\A(I_2')Z^* \simeq \A(I_1)\otimes Z\A(I_2')Z^*$ intertwining the gradings.
\end{lemma}

\begin{proof}
By our standing assumption and the Reeh-Schlieder property, the inclusion\linebreak $(\A(I_1)\subset\A(I_2),\Omega)$ is a standard split inclusion in the sense of \cite{DL2,DL84}, with $\Omega$ the (unique) vacuum vector of $\A$. By \cite[Sect.3]{DL2} (see also \cite{DL84}) there is a unique vector $\eta$ in the natural cone 
$P^\natural_\Omega(\A(I_1)'\cap \A(I_2))$ such that 
$\langle \eta, a_1 a_2' \eta \rangle = \langle \Omega, a_1 \Omega \rangle \langle \Omega, a_2' \Omega\rangle$, for all $a_1\in\A(I_1)$ and $a_2'\in\A(I_2)'$. Moreover, by the uniqueness of $\eta$ we have $\Gamma \eta = \eta$ (\cf the proof of \cite[Lemma 3.3]{DL2}). Then the unitary $V\in B(\H\otimes\H,\H)$ defined by 
$V (a_1\Omega \otimes a_2'\Omega)= a_1a_2'\eta$ satisfies  $V(a_1\otimes a_2') V^* = a_1 a_2'$ for all $a_1 \in \A(I_1)$ and all $a_2' \in \A(I_2)'$, and $V(\Gamma\otimes \Gamma)V^*=\Gamma$. Since, by graded Haag duality (3) in Section \ref{sec:CFT}, we have $\A(I_2)'= Z\A(I_2')Z^*$, the claim follows.
\end{proof}

\begin{proposition}\label{prop:gen2-locAA}
We have:
\begin{itemize}
\item[$(1)$] If $\Delta\subset\Delta^1$, then
\[
 \AA_\Delta(I) = \pi_R^{-1}(\dom(\delta))\cap \A^\gamma(I) \not= \C, \quad I\in\I.
\]
In particular, the local algebras $\AA_\Delta(I)$ are non-trivial and coincide with $\AA_{\Delta^1}(I)$.
\item[$(2)$] Given a unitary $u\in\A^\gamma(I)$, for some $I\in\I$, such that $\Ad(u) \in \Delta^1_I$, then $u\in\AA_{\Delta^1}(I)$.
\end{itemize}
\end{proposition}

\begin{proof}
(1) Given $x\in\pi_R^{-1}(\dom(\delta))\cap \A^\gamma(I)$ and $\rho\in\Delta_{I_0}$ with some $I_0\in\I$, choose $h\in\PSI$ such that $\dot{h}\bar{I}_0\cap \bar{I}=\emptyset$. Then
\[
 \rho(x) = \Ad(z(\rho,h))\circ \alpha_h\circ\rho\circ\alpha_h^{-1} (x) =\Ad(z(\rho,h))(x)\in\pi_R^{-1}(\dom(\delta)),
\]
by assumption on $x$ and $\Delta_{I_0}$. Since this holds for every $I_0\in\I$ and $\rho\in\Delta_{I_0}:= \Delta \cap \Delta^0_{I_0}$, we have $x\in\AA_\Delta(I)$. 
From Proposition \ref{prop:CFT-supersymmetric} and \cite[Thm.4.13]{CHKL} it follows that the closure of 
$\pi_R^{-1}(\dom(\delta))\cap \A^\gamma(I)$ in the $\sigma$-weak topology of $\A^\gamma(I)$ contains $\A^\gamma_{\SVirc}(I)$, 
the even part of the local algebra corresponding to the super-Virasoro subnet $\A_{\SVirc} \subset \A$ as in Definition \ref{def:CFT-SCFTnet}.
In particular $\pi_R^{-1}(\dom(\delta))\cap \A^\gamma(I)$ is nontrivial.
 
(2) Since $u\in\A^\gamma(I)$, $\Ad(u)$ defines an automorphism of $C^*(\A^\gamma)$ localized in $I$, and by assumption $\Ad(u)\in\Delta^1_I$. The unique inner representation establishing covariance is given by $U_{\pi^\gamma_0 \circ \Ad(u)}(g) := u U^\gamma (g) u^*$.
It follows that the corresponding cocycle is given by $z(\Ad(u), g) = u \alpha_g(u)^*$.
Next, choose $g\in {\mathcal U}_I$ and $I_0\supset \bar{I}$ such that $\dot{g} \bar{I_0}\cap \bar{I_0}=\emptyset$, and fix $I_1\in\I$ containing both $I_0$ and $\dot{g}I_0$. Then Lemma \ref{lem:gen2-locAA} gives us a unitary $V\in B(\H\otimes\H,\H)$ implementing the isomorphism $\A(I_0)\vee Z\A(\dot{g}I_0)Z^* \simeq \A(I_0)\otimes Z\A(\dot{g}I_0)Z^*$. Recall from the proof of Proposition \ref{prop:indpR} that, for a local element (like $u\alpha_g(u)^*$) in $\A^\gamma(I_1)$, it suffices to check differentiability in the vacuum representation, \ie whether it lies in the domain of $\delta_{I_1} = [G(\phi_{I_1}),\cdot]$; here $\phi_{I_1}\in\Cci(\S)$ is $1$ on $I_1$ and its support is non-dense in $\S$. Then on $\A^\gamma(I)\vee \A^\gamma(\dot{g}I)$, $\delta_{I_1}$ is implemented by the commutator with $G(\phi_I) + G(\phi_{\dot{g}I})=G(\phi_I) -\rmi\Gamma Z G(\phi_{\dot{g}I})Z^*$. We use the fact that $G(\phi_I)$ preserves $\Cci(L_0)$ and that it is affiliated with $\A(I_0)$. Moreover, by assumption $\delta_{I_1} (u \alpha_g(u)^*)$ is bounded. Thus, for $\psi_1,\psi_2,\phi_1,\phi_2\in\Cci(L_0)$ and using $V(\Gamma\otimes \Gamma)V^*=\Gamma$ from Lemma \ref{lem:gen2-locAA}, we obtain
\begin{align*}
\langle \psi_1\otimes\psi_2, V^* \delta_{I_1} (u \alpha_g(u)^*) V & (\phi_1\otimes\phi_2)\rangle
= \langle \psi_1\otimes\psi_2, V^* \big( (G(\phi_I)-\rmi\Gamma ZG(\phi_{gI})Z^*) u \alpha_g(u)^*\\
& - u \alpha_g(u)^* (G(\phi_I) -\rmi\Gamma ZG(\phi_{gI})Z^*) \big) V (\phi_1\otimes\phi_2)\rangle\\
=& \langle \psi_1\otimes\psi_2, [(G(\phi_I)\otimes\unit), (u \otimes \alpha_g(u)^*)] (\phi_1\otimes\phi_2)\rangle\\
&-\rmi \langle \psi_1\otimes\psi_2, [(\Gamma\otimes \Gamma)(\unit\otimes ZG(\phi_{gI})Z^*), (u \otimes \alpha_g(u)^*)] (\phi_1\otimes\phi_2)\rangle\\
=& \langle \psi_1\otimes\psi_2, ([G(\phi_I),u] \otimes \alpha_g(u)^*) (\phi_1\otimes\phi_2)\rangle\\
&-\rmi\langle \psi_1\otimes\psi_2, (\Gamma u \otimes \Gamma Z[G(\phi_{gI}),\alpha_g(u)^*]Z^*) (\phi_1\otimes\phi_2)\rangle\\
=& \langle \psi_1\otimes\psi_2, (\delta_{I_1}(u) \otimes \alpha_g(u)^*) (\phi_1\otimes\phi_2)\rangle\\
&+\langle \psi_1\otimes\psi_2, (\Gamma u \otimes \delta_{I_1}(\alpha_g(u)^*)) (\phi_1\otimes\phi_2)\rangle,
\end{align*}
so $\delta_{I_1}(u)\otimes \alpha_g(u)^*+ \Gamma u \otimes \delta_{I_1}(\alpha_g(u)^*)$ gives rise to a bounded operator on $\H\otimes\H$. Suppose $\delta_{I_1}(u)$ were unbounded; then we could find normalized sequences $\psi_{1,n},\phi_{1,n}\in \Cci(L_0)$, $n\in\N$, with fixed $\psi_2,\phi_2$, such that 
\[
  |\langle \psi_{1,n}\otimes\psi_2, (\delta_{I_1}(u) \otimes \alpha_g(u)^*) (\phi_{1,n}\otimes\phi_2)\rangle | \ra + \infty,
\]
while $\langle \psi_{1,n}\otimes\psi_2, (\Gamma u \otimes \delta_{I_1}(\alpha_g(u)^*)) (\phi_{1,n}\otimes\phi_2)\rangle$ remains bounded because $\|\Gamma u\|=1$ owing to unitarity; thus the sum of these two expressions would go to infinity, so $\delta_{I_1}(u)\otimes \alpha_g(u)^*+ \Gamma u \otimes \delta_{I_1}(\alpha_g(u)^*)$ would be unbounded, which is a contradiction. Hence, both $u$ and $\alpha_g(u)^*$ have to be in $\dom(\delta_{I_1})$, thus in $\pR^{-1}(\dom(\delta))$, and according to part (1), $u\in\AA_\Delta(I)$.
\end{proof}

\subsection*{The cocycles}

Let $\Delta \subset \Delta^0$. Then, for every $\rho \in \Delta$  we consider the $\theta$-summable spectral triple \linebreak
$(\AA_\Delta,(\pR \circ \rho,\H_R), Q)$ from Theorem \ref{theorem:gen2-AAST}, which is even when $\pR$ is graded and odd when it is ungraded. From the same theorem we also know that the representation $\pR\circ \rho: \AA_\Delta \to \dom(\delta)$ is continuous. Hence, by Theorem \ref{th:JLO} the spectral triple has a corresponding JLO cocycle which we will denote by $\tau_\rho$. The JLO cocycles 
$\tau_\rho$, $\rho \in \Delta$, will play a central role in the rest of this article as noncommutative geometric invariants associated to 
DHR endomorphisms. Note that if $\Delta$ is a semigroup then, for $\rho, \sigma \in \Delta$, we have $\rho\sigma\in\Delta$ and
$\tau_{\rho\sigma} = {\sigma^*} \tau_\rho$ where $\sigma^*$ is the pull-back of $\sigma \in \End(\AA_\Delta)$ to $\End(HE^*(\AA_\Delta))$.
In particular, $\tau_\rho = {\rho^*}\tau_{\id}$, for all $\rho \in \Delta$.  

In general, if $\rho, \sigma \in \Delta$ are localized in a given interval $I$ and $[\rho] = [\sigma]$, then there is a unitary $u \in \A^\gamma(I)$ such that $\sigma = \Ad(u)\rho$. Accordingly, if 
$\pR(u) \in \dom(\delta)$, then $\tau_\sigma = \tau_\rho^{\pR (u)}$ and hence 
$[\tau_\rho] = [\tau_\sigma]$ by Proposition \ref{prop:perturb-v}.
 However, in general, the cohomology class of the cocycle $\tau_\rho$ associated to $\rho$ could be different from the one of the cocycle 
 $\tau_\sigma$ associated to a localized endomorphism $\sigma$ equivalent to $\rho$. Therefore, the cohomology classes of the cocycles $\tau_\rho$, $\rho \in \Delta$, need not give invariants for DHR endomorphisms in the strict sense but only with respect to a finer equivalence relation involving the differentiability of intertwiners. One could say that they are not ``topological'' invariants but only invariants for the ``differentiable structure''.  Nonetheless, as a consequence of the following proposition, this distinction turns out to be unnecessary in the case of differentiably transportable automorphisms.
 
\begin{proposition}\label{prop:gen2-tau1}
Suppose $\Delta\subset\Delta^1$. Given $I\in\I$ and two automorphisms $\rho,\sigma \in\Delta_I$ which are equivalent via a unitary in $\A^\gamma(I)$, then the unitary lies actually in $\AA_\Delta(I)$ and the associated cocycles ${\tau_\rho}$ and ${\tau_\sigma}$ over $\AA_\Delta$ give rise to the same cohomology class.
\end{proposition}

\begin{proof}
We start with the following simple observation for localized covariant \emph{auto}morph\-isms, a consequence of Proposition \ref{prop:gen2-invert}:
\begin{equation}\label{eq:gen2-buy2get3}
 \textrm{if two out of $\sigma,\rho,\sigma\rho$ are in $\Delta^1_I$ then the third one lies in $\Delta^1_I$.}
\end{equation}
Let now $\rho$ and $\sigma$ be the two equivalent automorphisms and $u\in\A^\gamma(I)$ the intertwining unitary.
Considering the three localized automorphisms $\Ad(u),\rho, \sigma=\Ad(u)\rho$, the latter two are in $\Delta_I$ by assumption, so $\Ad(u)\in\Delta_I^1$ by \eqref{eq:gen2-buy2get3}. Then Proposition \ref{prop:gen2-locAA}  implies that $u\in\AA_{\Delta^1}(I)=\AA_\Delta(I)\subset \AA_\Delta$. Applying Proposition \ref{prop:perturb-v} with $A=\AA_\Delta$, $\pi=\pR\circ \rho$ and $v=\pR(u)$ finally provides the equivalence of the cocycles ${\tau_\sigma} = \tau_{\Ad(u) \rho}=\tau_\rho^v$ and $\tau_\rho$ on the global algebra $\AA_\Delta$.
\end{proof}

Our final observation has already been announced at the beginning of this section, namely that our spectral triples have to be constructed out of a certain global rather than local algebra:

\begin{proposition}\label{prop:gen2-CT-eq}
Given any $\rho\in \Delta \cap \Delta^1$, we have
\[
[{\tau_\rho}| _{\AA_\Delta(I)}] = [{\tau_{\id}} |_{\AA_\Delta(I)}],\quad I\in\I.
\]
In other words, all JLO cocycles for differentiably transportable endomorphisms are locally cohomologous.
\end{proposition}

\begin{proof}
Let $I_1\in\I$ be the localization region of $\rho$ and let $I\in\I$ be given. Choose any $I_2\in\I$ such that $\bar{I_2}\subset I'$, and choose 
$g\in\PSI$ such that $I_2=\dot{g}I_1$. Then $\rho^g=\Ad (z(\rho,g)^*)\circ \rho$ is localized in $I_2$ and hence ${\rho^g}$ acts trivially on 
$\AA_\Delta(I)$. It follows that ${\tau_\rho} |_{\AA_\Delta(I)} ={\tau_{\id}^{\pR(z(\rho,g))}}| _{\AA_\Delta(I)} $. Now,  according to Proposition 
\ref{prop:perturb-v}, $[{\tau_{\id}^{\pR(z(\rho,g))}}] = [\tau_{\id}]$ and hence there exists an entire cochain $\psi\in CE^\bullet(\AA_{\Delta})$ such that 
${\tau_{\Ad(z(\rho,g))}} = \tau + {\partial} \psi$. 
Restricting a coboundary to a unital subalgebra gives us again a coboundary: a coboundary is the image under ${\partial}$ of a cochain, and restricting a cochain to a subalgebra gives again a cochain, now over the subalgebra; applying then ${\partial}$ to this restricted cochain defines a coboundary (over the subalgebra), which by construction is just the restriction of the original coboundary.
Thus $({\partial}\psi)|_{\AA_{\Delta}(I)} = {\partial}(\psi |_{\AA_{\Delta}(I)})\in CE^\bullet(\AA_{\Delta}(I))$ is again a coboundary and 
${\tau_\rho}|_{\AA_{\Delta}(I)} = \tau_{\id} |_{\AA_{\Delta}(I)}+ ({\partial} \psi)|_{\AA_{\Delta}(I)}$ is cohomologous to $\tau_{\id} |_{\AA_{\Delta}(I)}$.
\end{proof}

\section{Pairing with K-theory for superconformal nets and geometric invariants for DHR endomorphisms }\label{sec:gen-pairing}

\subsection*{Even case}

The spectral triples in Theorem \ref{theorem:gen2-AAST} associated to a superconformal net $\A$, a subset $\Delta\subset\Delta^0$ and an irreducible graded Ramond representation $\pR$ of $\A$ is even, so the corresponding JLO cocycles $\tau_\rho$, $\rho \in \Delta$, are even, and they pair with 
$K_0(\AA_\Delta)$-classes: in fact, according to Theorem \ref{th:JLO}(3), for a projection $p\in\AA_\Delta$, the densely defined operator $\pR(\rho(p))_- Q_+ \pR(\rho(p))_+$ from $\pR(\rho(p))_+\H_{R,+}$ to $\pR(\rho(p))_-\H_{R,-}$ is a Fredholm
operator and  
\[
 \tau_\rho(p) =  \ind_{\pR(\rho(p))_+\H_{R,+}} (\pR(\rho(p))_- Q_+ \pR(\rho(p))_+),
\]
depending only on the class of $p$ in $K_0(\AA_\Delta)$.

The task is to choose this $p$ in a suitable and general (model-independent) manner. Note that if $\pR(p)$ projects onto a subspace of 
$\H_{R,+}$, then $\pR(p)_- =0$ so that $\pR(p)_-Q_+ \pR(p)_+ = 0$ and  $\pR(p)_-\H_{R,-} = \{0\}$, and it follows that 
\begin{equation}\label{eq:gen2-p-even}
\begin{aligned}
\tau_{\id}(p) =& \ind(\pR(p)_- Q_+ \pR(p)_+) = \dim \ker_{\pR(p)_+\H_{R,+}} (0) - \dim \ker_{\pR(p)_-\H_{R,-}} (0)\\
 =& \dim (\pR(p)\H_{R,+}).
\end{aligned}
\end{equation}
However, not all such $p$ are in $\AA_\Delta$, \eg the projection $\frac12(\unit + \Gamma_R)\not\in \dom(\delta)$ onto $\H_{R,+}$, so we have to find suitable subprojections.

Given a representation $\pi$ of $W^*(\A^\gamma)$, which is quasi-equivalent to a subrepresentation of the universal representation of $W^*(\A^\gamma)$, denote by $s(\pi)\in \ZZ(W^*(\A^\gamma))$ the central support of the projection onto this subrepresentation so that, in 
particular $\pi(s(\pi)) = \unit$. Recall from Proposition \ref{prop:CKL22} and Remark \ref{remark:gradedpR} the decomposition $\pR |_{C^*(\A^\gamma)} = \pi_{R,+}\oplus \pi_{R,-}=\pi_{R,+}\oplus \pi_{R,+}\circ \hat{\gamma}$ into irreducible representations of $C^*(\A^\gamma)$ on $\H_R = \H_{R,+}\oplus \H_{R,-}$. Considering then for $\pi$ the irreducible representation $\pi_{R,+}$, we get $\pi_{R,+}(W^*(\A^\gamma))= B(\H_{R,+})$, and $\pi_{R,+}$ restricts to an isomorphism $s(\pi_{R,+})W^*(\A^\gamma)\ra B(\H_{R,+})$. 
Let $\H_{R,0} \subset \H_R$ be the minimal energy subspace,
\ie the (finite-dimensional) $L_0^\pR$-eigenspace corresponding to the eigenvalue $\lw(\pR)$, and let 
$\H_{R,0,+}:= \H_{R,0} \cap \H_{R,+}$ be its even part. Note that, by replacing if necessary $\Gamma_R$ with $-\Gamma_R$, we can always 
assume that $\dim(\H_{R,0,+}) \neq 0$, so that $\lw(\pR)=\lw(\pi_{R,+})$.
 We consider any projection $p\in W^*(\A^\gamma)$ such that 
$\pi_{R,+}(p)\in B(\H_{R,+})$ is the projection onto $\H_{R,0,+}$. Then $p_{0,+}:= p \cdot s(\pi_{R,+})\in W^*(\A^\gamma)$ is well-defined and does not depend on the explicit choice of $p$, and neither does it depend on the explicit $\pR$ in its unitary equivalence class.

\begin{definition}\label{def:gen2-p0+}
Given a graded Ramond representation $(\pR,\H_R,\Gamma_R)$ of $\A$, we call the associated projection $p_{0,+} \in W^*(\A^\gamma)$ its \emph{characteristic projection}.
\end{definition}

It is natural to ask whether $p_{0,+}$ lies even in $\AA_\Delta$, and whether it does the job we want it to do. A partial answer is contained in

\begin{proposition}\label{prop:gen2-special-p0+}
\begin{itemize}
\item[$(1)$] $p_{0,+}$ has the following characteristic property:
\[
 \pi_{R,+}(p_{0,+})=\textrm{projection onto $\H_{R,0,+}$}
\]
and $\pi(p_{0,+}) = 0$ if $\pi_{R,+}$ is not equivalent to a subrepresentation of $\pi$.
\item[$(2)$] In $(\pR,\H_R)$, the characteristic projection attains the form 
\begin{equation}\label{eq:gen2-special-p0+}
 \pR(p_{0,+}) = \chi_1(\rme^{-(L_0^\pR-\lw \pR)}) \frac{\unit + \Gamma_R}{2} \in\dom(\delta)\subset B(\H_R),
\end{equation}
where $\chi_1$ denotes the characteristic function of $\{1\}\subset \R$.
\item[$(3)$] Suppose $\Delta\subset\Delta^1$, and $\pR\circ \rho$ and $\pR$ are disjoint, for every $\rho\in\Delta$ with $[\rho]\not= [\id]$. Then $p_{0,+}\in\AA_\Delta$.
\end{itemize}
\end{proposition}

\begin{proof}
(1) is obvious from the definition of $p_{0,+}$.

(2) Since $\pi_{R_+}$ and $\pi_{R_-}$ are inequivalent irreducible representations they are disjoint and hence 
$\pi_{R_-} (s(\pi_{R,+})) =0$; recalling that $\pi_{R_+} (s(\pi_{R,+})) = \unit$ it follows that $\pR (s(\pi_{R,+}))$ is the projection of 
$\H_R$ onto $\H_{R,+}$, i.e., $\frac12(\unit+\Gamma_R)$. The projection in $B(\H_R)$ onto the finite-dimensional subspace of lowest energy $\H_{R,0}$ is obviously given by $\chi_1(\rme^{-(L_0^\pR-\lw \pR)})$, which proves our claim.

(3) By definition, $\pR(p_{0,+})$ is the projection onto the even part of the \emph{finite}-dimensional eigenspace  corresponding to the  
$L_0^\pR$-eigenvalue $\lw(\pR)$.  $Q$ commutes with $L_0^\pR$ and hence it restricts to an operator, which is obviously bounded, on this finite-dimensional eigenspace (with spectrum $\subset\{\pm\sqrt{\lw(\pR)-c/24}\}$). It follows that $\pR(p_{0,+})\in\dom(\delta)$. 
Given $\rho \in\Delta$ equivalent to $\id$, we have $\rho= \Ad(u)$ with $u$ a unitary in $\AA_\Delta$, according to Proposition \ref{prop:gen2-CT-eq}; thus $\pR\circ \rho(p_{0,+})= \pR(u)\pR(p_{0,+})\pR(u)^*$ lies in $\dom(\delta)$, too. 
On the other hand, for $\rho\in\Delta$ with $\rho \not\simeq \id$, we have by assumption $\pR \circ \rho(p_{0,+})=0$, so 
$\pR\circ \rho(p_{0,+})\in\dom(\delta)$.
\end{proof}

Now we are in the position to prove our main theorem concerning the even index pairing:

\begin{theorem}\label{th:gen2-pairing-even}
Let $\A$ be a superconformal net with a fixed supersymmetric graded Ramond irreducible representation $(\pi_R,\H_R)$ with supercharge $Q$, and let $\Delta\subset\Delta^0$ be a subset of localized endomorphisms with $\id \in \Delta$. Then $(\AA_\Delta, (\pR\circ \rho,\H_R),Q)_{\rho\in\Delta}$ is a family of even $\theta$-summable spectral triples and the associated even JLO cocycles have the following properties:
\begin{itemize}
\item[$(1)$] Suppose $\tau_{\id}(\unit) = \ind (Q_+) \not= 0$ and that, for fixed $\sigma\in\Delta$ and all $\rho\in\Delta$ with $[\rho]\not=[\sigma]$, $\pR\circ \rho$ and $\pR\circ \sigma$ are disjoint. Then, for all $\rho\in\Delta$ with $[\rho]\not=[\sigma]$, we have
$[{\tau_\rho}]\not=[{\tau_\sigma}]$.
\item[$(2)$] Suppose that, for fixed automorphism $\sigma\in\Delta$ and all $\rho\in\Delta$ with $\rho\not=\sigma$, $\pR \circ \rho$ and 
$\pR\circ \sigma$ are disjoint. Then for every $\rho\in \Delta$ with $\rho\not=\sigma$, we have $[{\tau_\rho}]\not=[{\tau_\sigma}]$.
\item[$(3)$] Suppose $\Delta\subset\Delta^1$ and that, for fixed automorphism $\sigma\in\Delta$ and all $\rho\in\Delta$ with 
$[\rho]\not=[\sigma]$, $\pR \circ \rho$ and $\pR \circ \sigma$ are disjoint. Then for every $\rho\in \Delta$, we have
\[
 [\rho]=[\sigma] \quad \textrm{iff} \quad [{\tau_\rho}]=[{\tau_\sigma}].
\]
\end{itemize}
In either case, the two non-equivalent cocycles are separated by pairing them with a suitable element from $K_0(\AA_\Delta)$.
\end{theorem}

\begin{proof}
We have an even $\theta$-summable spectral triple according to Theorem \ref{theorem:gen2-AAST}.

(1) Note that under the present conditions, $\pR \circ \rho(s(\pR\circ\sigma)) = 0$ if $[\rho]\not=[\sigma]$, and \linebreak $\pR \circ 
\rho(s(\pR \circ \sigma))= \unit$ if $[\rho]=[\sigma]$; thus $s(\pR \circ \sigma)\in\AA_\Delta$. Moreover, it separates the cocycles since ${\tau_\rho}(s(\pR \circ \sigma)) = 0$ if $[\rho]\not=[\sigma]$, while ${\tau_\rho}(s(\pR \circ \sigma))= \tau_{\id}(\unit)=\ind_{\H_{R,+}}(Q_+) \not= 0$ if $[\rho]=[\sigma]$, and this depends only on the class $[s(\pR \circ \sigma)]\in K_0(\AA_\Delta)$, according to Theorem \ref{th:JLO}(3).

(2) Suppose $\rho\not=\sigma$. Then
\[
\pR\circ \rho(\sigma^{-1}(p_{0,+})) = \pR \circ \rho\sigma^{-1}(p_{0,+}) = 0
\]
because $p_{0,+}< s(\pR)$ and because $\pR\circ \rho$ and $\pR\circ \sigma$ are disjoint by assumption, whence 
$\pR\circ \rho\sigma^{-1}$ and $\pR$ are disjoint since $\sigma$ is an automorphism. On the other hand, for $\rho=\sigma$, 
\[
 \pR \circ \rho(\sigma^{-1}(p_{0,+}))= \pR(p_{0,+}) \in \dom(\delta),
\]
according to the preceding proposition, so $\sigma^{-1}(p_{0,+})\in\AA_\Delta$.  If $\rho\not=\sigma$, we then have ${\tau_\rho}(\sigma^{-1}(p_{0,+}))=0$, whereas ${\tau_\sigma}(\sigma^{-1}(p_{0,+}))= \dim (\H_{R,0,+})\not= 0$, so that $[{\tau_\rho}]\not=[{\tau_\sigma}]$, and they are separated by (the $K_0(\AA_\Delta)$-class of) $\sigma^{-1}(p_{0,+})$.

(3) Since $\Delta\subset\Delta^1$, $\rho$ is equivalent via a unitary in $\AA_\Delta$ to an endomorphism localized in the same interval $I\in\I$ as $\sigma$, so we may assume without loss of generality that $\rho$ and $\sigma$ are actually localized in the same interval $I$. If $[\rho]=[\sigma]$, then according to Proposition \ref{prop:gen2-CT-eq} they are intertwined by a unitary $u\in\AA_\Delta(I)$ wherefore
\[
 \pR\circ \rho(\sigma^{-1}(p_{0,+})) = \Ad(\pR(u))(\pR(p_{0,+})) \in \dom(\delta).
\]
On the other hand, if $[\rho]\not=[\sigma]$, then $\pR\circ \rho(\sigma^{-1}(p_{0,+}))=0$ by the assumption on disjointness. Thus $\sigma^{-1}(p_{0,+})\in\AA_\Delta$, and it separates the cocycles as follows:

\[
 {\tau_\rho}(\sigma^{-1}(p_{0,+})) = 
\left\lbrace\begin{array}{l@{\; \;\textrm{if} \; }l}
\tau_{\id}^{\pR(u)} (p_{0,+}) = \tau_{\id} (p_{0,+}) =  \dim (\H_{R,0,+}) & [\rho]=[\sigma] \\
0 & [\rho] \neq [\sigma ],
\end{array}\right.
\]
and according to Theorem \ref{th:JLO}(3) this depends only on the class of $\sigma^{-1}(p_{0,+})$ in $K_0(\AA_\Delta)$.
\end{proof}

\subsection*{Odd case}

The spectral triple associated to a superconformal net $\A$, a subset $\Delta\subset \Delta^1$ and an irreducible ungraded Ramond representation $\pR$ is odd. According to Theorem \ref{th:JLO}(3) the index pairing for odd spectral triples is given by
\[
 \tau_{\id} (v) = \ind_{\chi_{[0,\infty)}(Q)\H_R}(\chi_{[0,\infty)}(Q) \pR(v) \chi_{[0,\infty)}(Q)),
\]
for every unitary $v\in\AA_\Delta$, where $\chi_{[0,\infty)}(Q) \pR(u_{0,+}) \chi_{[0,\infty)}(Q)$ is considered as an operator from 
${\chi_{[0,\infty)}(Q)\H_R}$ to ${\chi_{[0,\infty)}(Q)\H_R}$, and this depends only on the class $[v]\in K_1(\AA_\Delta)$. Pictorially speaking, the unitary $\pR(v)$ should therefore act as a certain shift on the spectrum of $Q$ if we want a non-vanishing pairing, or otherwise as the unit element if we want a trivial pairing. Note here that the spectrum of $Q$ is discrete because that of $L_0^\pR=Q^2+\frac{c}{24}{\unit}$ is so. 

Since $\pR$ is ungraded, it remains irreducible in restriction to $\A^\gamma$ (\cf Proposition \ref{prop:CKL22}) and defines an irreducible representation of $W^*(\A^\gamma)$, denoted again $\pR$. As in the even case, $s(\pR)\in \ZZ(W^*(\A^\gamma))$ stands for the central support of the projection onto the subrepresentation $\pR$ of the universal representation of  $W^*(\A^\gamma)$, and $\pR$ restricts to an isomorphism $\pR |_{s(\pR)W^*(\A^\gamma)}$ between $s(\pR)W^*(\A^\gamma)$ and $B(\H_R)$ owing to irreducibility. Thus
\[
 u_{0,+} := (\pR |_{s(\pR)W^*(\A^\gamma)})^{-1}(u_s) + (\unit - s(\pR)) \in W^*(\A^\gamma),
\]
with $u_s\in\dom(\delta)\subset\B(\H_R)$ the spectrum shift unitary from Construction \ref{con:gen2-shift} below depending on $Q$, is a well-defined unitary.

\begin{definition}\label{def:gen2-u0+}
Given an ungraded Ramond representation  $(\pR,\H_R)$ of $\A$, we call the above $u_{0,+}\in W^*(\A^\gamma)$ the \emph{characteristic unitary} for $\pR$. 
\end{definition}

We have a similar characteristic property as in the even case (\cf Proposition \ref{prop:gen2-special-p0+}): $\pR(u_{0,+})= u_s$, the spectrum shift, while $\pi(u_{0,+}) =\unit$ if $\pi$ has no subrepresentation equivalent to $\pR$.

\begin{construction}\label{con:gen2-shift}
We would like to give a general construction of the \emph{spectrum shift unitary} $u_s\in B(\H_R)$ used in the preceding definition. The construction is somehow lengthy and technical, but the point is that a priori not much is known about the dimension of the eigenspaces of $Q$, so we have to go through all single steps.

(1) Consider the lowest energy subspace of $\H_R$: the finite-dimensional subspace where $L_0^\pR$ has eigenvalue $h=\lw(\pR)$. On this subspace, $Q=G_0^\pR$ is diagonalizable with spectrum $\subset \{\pm \sqrt{h-\frac{c}{24}} \}$, since $Q$ is selfadjoint and $Q^2=L_0^\pR-\frac{c}{24} \unit$. Let $p\in B(\H_R)$ be the projection onto the (nonzero!) eigenspace of $Q$ corresponding to the eigenvalue 
$\lambda_0:= \sqrt{h-\frac{c}{24}}$, denoted here by $\H_{R,0,+}$ but not to be confused with the different one in the even pairing. Then $p$ is well-defined, non-trivial, lies  in the image $\pR(W^*(\A^\gamma))$ owing to irreducibility, and commutes with $Q$ by construction.

(2) For $n\in\N$, let $\H_{R,n}$ be the subspace of $L_0^\pR$-eigenvalue $h+n$, and consider two orthogonal copies $\H_{R,n,\pm}$ of $\H_{R,0,+}$ in $\H_{R,n}$ such that the selfadjoint and diagonalizable operator $Q$ has eigenvalue $\lambda_n:=\sqrt{h+n-\frac{c}{24}}>0$ on $\H_{R,n,+}$ and $-\lambda_{n}<0$ on $\H_{R,n,-}$. Let us check that this is always possible. First, given $\xi_0\in\H_{R,0,+}$ such that $Q\xi_0=\lambda_0\xi_0$, let
\[
\xi_{\pm n} :=\Big(L_{-n}^\pR + \frac12 \Big(-\lambda_0 \pm\sqrt{\lambda_0^2+n}\Big) G_{-n}^\pR\Big)\xi_0.
\]
To understand this definition, recall from Definition \ref{def:CFT-SCFTnet} that $\A$ as a superconformal net contains the super-Virasoro net introduced in Example \ref{ex:superVir} as a conformal subnet. Then $\pR$ restricts to a Ramond representation of that subnet, and according to Proposition \ref{prop:CFT-supersymmetric} we have the corresponding field operators $G_n^\pR$ and $L_n^\pR$ acting on $\H_R$. Using the commutation relations \eqref{eq:superVir}, which hold in particular on the finite energy vectors like $\xi_{\pm n}$, we then find
\begin{align*}
Q \xi_{\pm n} =& Q \Big(L_{-n}^\pR + \frac12 \Big(-\lambda_0 \pm\sqrt{\lambda_0^2+n}\Big)G_{-n}^\pR\Big)\xi_0 \\
=& \lambda_0 \Big(L_{-n}^\pR - \frac12 \Big(-\lambda_0 \pm\sqrt{\lambda_0^2+n}\Big)G_{-n}^\pR\Big)\xi_0 + \Big(\frac{n}{2}G_{-n}^\pR + \Big(-\lambda_0 \pm\sqrt{\lambda_0^2+n}\Big) L_{-n}^\pR\Big)\xi_0\\
=& \pm \sqrt{\lambda_0^2+n} L_{-n}^\pR \xi_0  + \frac12 \Big(\lambda_0^2 + n \mp \lambda_0 \sqrt{\lambda_0^2+n}\Big)G_{-n}^\pR\xi_0\\
=& \pm \sqrt{\lambda_0^2+n}\Big(L_{-n}^\pR + \frac12 \Big(-\lambda_0 \pm\sqrt{\lambda_0^2+n}\Big) G_{-n}^\pR \Big) \xi_0\\
=& \pm \sqrt{\lambda_0^2+n}\xi_{\pm n}.
\end{align*}
So for every $n\in\N$, the corresponding $\xi_{\pm n}\in\H_{R,n}$ are eigenvectors with eigenvalues $\pm\lambda_n = \pm \sqrt{\lambda_0^2+n}$, respectively. Second, given two such $\xi_0 \perp \eta_0 \in\H_{R,0}\subset \ker(L_n^\pR)\cap \ker(G_n^\pR)$, with $\lambda_0\in\R$ the $Q$-eigenvalue of $\eta_0$ and arbitrary $\alpha\in\R$, we have
\begin{align*}
\langle (L_{-n}^\pR+\alpha G_{-n}^\pR)\xi_0,&(L_{-n}^\pR+\alpha G_{-n}^\pR)\eta_0 \rangle 
= \langle \xi_0, (L_{n}^\pR+\alpha G_{n}^\pR)(L_{-n}^\pR+\alpha G_{-n}^\pR)\eta_0 \rangle\\
=& \Big\langle \xi_0, \Big( 2(n+\alpha^2)L_0 + 3n\alpha G_0^\pR + (n^3 +4n^2 -n -1)\frac{c}{12}\unit \Big) \eta_0 \Big\rangle\\
=& \Big( 2(n+\alpha^2)(\lambda_0^2+\frac{c}{24}) + 3n\alpha \lambda_0 + (n^3 +4n^2 -n -1)\frac{c}{12} \Big) \langle \xi_0,\eta_0 \rangle
\end{align*}
which vanishes due to the orthogonality assumption $\xi_0 \perp \eta_0$, so $\xi_{\pm n}$ and $\eta_{\pm n}$ are again mutually orthogonal. These two facts together show that we have two copies of $\H_{R,0,+}$ in $\H_{R,n}$ on which $Q$ has eigenvalues $\lambda_{n}>0>\lambda_{-n}$ respectively, so they are in fact orthogonal and unambiguously denoted by $\H_{R,n,\pm}$.

(3) Consider now
\[
\H_{R,\operatorname{shift}}:= ...\oplus\H_{R,n,-}\oplus...\oplus \H_{R,1,-}\oplus\H_{R,0,+}\oplus\H_{R,1,+}\oplus...\oplus\H_{R,n,+}\oplus ....
\]
Let $u_s$ be the standard left shift on $\H_{R,\operatorname{shift}}$ (mapping every component isomorphically into the next one on its left, in particular $\xi_n\mapsto \xi_{n-1}$, for every $n\in\Z$ and $\xi_0\in\H_{R,0,+}$), extended by the identical action on the orthogonal complement of $\H_{R,\operatorname{shift}}$ in $\H_R$. It is clearly a unitary in $B(\H_R)$. Moreover, it has bounded commutator with $Q$: for all $\xi_0\in\H_{R,0,+}$ and $n\in\N$, we have
\[
(u_sQ u_s^* - Q) \xi_{\pm n} = u_s Q \xi_{\pm n+1} \mp \lambda_n \xi_{\pm n} = \pm (\lambda_{n\pm 1}-\lambda_n) \xi_{\pm n}, 
\]
where $|(\lambda_{n\pm 1}-\lambda_n)| = |\sqrt{h+n\pm 1}-\sqrt{h+n}|\le 1$, and
\[
(u_sQ u_s^* - Q) \psi = 0, \quad \psi\in \H_R\ominus \H_{R,\operatorname{shift}},
\]
so $u_s^*\in\dom(\delta)$ and hence $u_s\in\dom(\delta)$.

(4) Now that we have the above spectrum shift $u_s\in\dom(\delta)$, the desired index on $\H_{R,\operatorname{shift}}$ is well-defined, and we find
\[
 \ind_{\chi_{[0,\infty )}(Q)\H_R}(\chi_{[0,\infty )}(Q) u_s \chi_{[0,\infty )}(Q)) = \dim \H_{R,0,+} \not= 0.
\]

(5) Let us make a brief remark. Performing the above construction in the case of a graded representation $(\pR,\H_R,\Gamma_R)$ has a serious consequence: the shift unitary $u_s\in B(\H_R)$ is not even with respect to the natural grading $\Gamma_R$, so it cannot lie in the image $\bigvee_{I\in\I}\pR(\A^\gamma(I))\simeq B(\H_{R,+})\oplus B(\H_{R,-})$, but only in $B(\H_R)=B(\H_{R,+}\oplus\H_{R,-})$. In the case of ungraded $\pR$, however, $\pR$ remains irreducible in restriction to $\A^\gamma$, so $\bigvee_{I\in\I}\pR(\A^\gamma(I)) =B(\H_R)$, and $u_s\in B(\H_R)$ therefore lies in the von Neumann algebra generated by the $\pR(\A^\gamma(I))$, $I\in\I$. \boxy
\end{construction}

We summarize this construction in

\begin{theorem}\label{th:gen2-pairing-odd}
Let $\A$ be a superconformal net with a fixed supersymmetric ungraded Ramond irreducible representation $(\pi_R,\H_R)$ with supercharge $Q$, and let $\Delta\subset\Delta^0$ be a subset of localized endomorphisms with $\id \in \Delta$. Then $(\AA_\Delta,(\pR \circ \rho,\H_R),Q)_{\rho\in\Delta}$ is a family of odd $\theta$-summable spectral triples and the associated odd JLO cocycles have the following properties:
\begin{itemize}
\item[$(1)$] Suppose that, for fixed automorphism $\sigma\in\Delta$ and all $\rho\in\Delta$ with $\rho\not=\sigma$, $\pR \circ \rho$ and 
$\pR \circ \sigma$ are disjoint. Then for every $\rho\in \Delta$ with $\rho\not=\sigma$, we have $[{\tau_\rho}]\not=[{\tau_\sigma}]$.
\item[$(2)$] Suppose $\Delta\subset\Delta^1$ and that, for fixed automorphism $\sigma\in\Delta$ and all $\rho\in\Delta$ with $[\rho]\not=[\sigma]$, $\pR \circ \rho$ and $\pR\circ \sigma$ are disjoint. Then for every $\rho\in \Delta$, we have
\[
 [\rho]=[\sigma] \quad \textrm{iff} \quad [{\tau_\rho}]=[{\tau_\sigma}].
\]
\end{itemize}
In either case, the two non-equivalent cocycles are separated by pairing them with a suitable element from $K_1(\AA_\Delta)$.
\end{theorem}

\begin{proof}
We have a family of odd $\theta$-summable spectral triples according to Theorem \ref{theorem:gen2-AAST} and our standing assumption \eqref{eq:gen2-trace-class}. The remaining statements (1)-(2) are proved as those in Theorem \ref{th:gen2-pairing-even}(2)-(3) with $p_{0,+}$ replaced by $u_{0,+}$ using the pairing of the cocycles with $K_1(\AA_\Delta)$ in Theorem \ref{th:JLO}(3).

\end{proof}

\section{Examples from super-current algebra nets and\\ super-Virasoro nets}\label{sec:loop}

\subsection*{Basics of the  super-current algebra net}

In order to define the super-current algebra net and its representations, we need basically two ingredients: loop group nets and free fermion nets, and we start with the former ones.

Let $G$ be a simple simply connected simply laced compact Lie group, $\lg$ its simple Lie algebra and $\Loop G=\Cci(\S,G)$ its loop group. Let $d$ denote the dimension, $h^\vee$ the dual Coxeter number, $\langle\cdot,\cdot\rangle$ the basic inner or scalar product of $\lg$, a multiple of the Killing form normalized in such a way that $\langle \theta,\theta\rangle=2$ with $\theta$ the highest root of $\lg$. Furthermore, let $(e_a)_{a=1,...,d}$ be an orthonormal basis with respect to this scalar product and $f_{abc}$ the structure constants of $\lg$ with respect to $(e_a)_{a=1,...,d}$, \cf \cite[Sect.1\&2]{Kac}. 

The corresponding \emph{affine Kac-Moody algebra} or \emph{$\lg$-current algebra} is the complex Lie algebra $\hat{\Loop} \lg$ generated by $J^a_{n}$, $a=1,...,d$ and $n\in\Z$, $\hat{c}_\lg$ and $\hat{d}_\lg$, with commutation relations
\begin{align*}
 [J^a_{m},J^b_{n}] =& \sum_{c} \rmi f_{abc} J^c_{m+n} + \delta_{m+n,0} \delta_{a,b} m \hat{c}_G,\quad [J^a_{m}, \hat{c}_\lg] = 0,\\
 [J^a_{m},\hat{d}_\lg]= & m J^a_{m}, \quad [\hat{c}_\lg,\hat{d}_\lg]= 0,
\end{align*}
\cf \cite{KT} and also \cite[Sect.7]{Kac} for a more systematic construction based on the central extension of the complexified loop algebra $\Cci(\S,\lg_\C)$. The above scalar product on $\lg$ extends to a scalar product on the complexification $\lg_\C$ and thus gives rise to a scalar product on $\Cci(\S,\lg_\C)$ namely $f,g\mapsto \langle f,g \rangle :=\frac{1}{2\pi} \int_\S \langle f,g\rangle$; we write $\K$ for the corresponding Hilbert space completion, which consists of $\lg_\C$-valued square-integrable functions on $\S$, and $\K_I$ for the subspace generated by those with compact support in $I\in\I$. An orthonormal basis of $\K$ is given by $e_r^a:= e_a \iota^r$, $a=1,...,d$, $r\in\frac12+\Z$, with the smooth functions $\iota^r:z\in\S\setminus\{-1\}\mapsto c_r z^r\in\S $ and $c_r$ a suitable normalization scalar. Another orthonormal basis is obtained analogously with the choice $r\in\Z$. Notice that $\lg_\C$ can be identified with the Lie subalgebra of $\Cci(\S,\lg_\C)$ spanned by $(e_0^a)_{a=1,...,d}$ or the Lie subalgebra of $\hat{\Loop} \lg$ generated by $(-\rmi J^a_0)_{a=1,...,d}$.

Second, the \emph{$d$-fermion algebra} or \emph{(self-dual) CAR algebra} $\FF$ is the unital graded C*-algebra generated by odd elements $F(f)$, with $f\in\K$, satisfying $F(f)^*=F(\bar{f})$ and the anticommutation relations $[F(\bar{f}),F(g)]_+=\langle f,g \rangle \unit$, so in particular $\|F(f)\|=\|f\|$, \cf \cite{Ara,Boc} for further information. The elements $F^a_r:= F(e^a_r)$, with $a=1,...,d$ and $r\in\frac12+\Z$ or $\Z$ and commutation relations
\[
 [F_r^a,F_s^b]_+ = \delta_{r+s,0}\delta_{a,b} \unit.
\]
define two Lie superalgebras inside $\FF$, called the \emph{Neveu-Schwarz fermion algebra} $\FF_{NS}$ and \emph{Ramond fermion algebra} $\FF_R$, respectively.

We are interested in unitary highest weight irreducible representations of $\lg$. These representations are characterized by a level $l\in\N$ (the scalar value of the central element $\hat{c}_\lg$ in the representation) and an integral dominant weight $\lambda$ of $\lg$ (which determines the action of the maximal toral subalgebra, $\mathfrak{t_g}$ of $\lg$ on the highest weight vector) such that $\langle \lambda,\theta\rangle \le l$, and are denoted here by $(\pi^\lg_{l,\lambda}, \H^\lg_{l,\lambda})$, \cf \cite[Sect.10\&11]{Kac}, and we shall use the same symbol for the corresponding Hilbert space completions. For each level $l \in \N$, the set $\Phi^\lg_l$ of the allowed integral dominant weights is finite and hence the set of equivalence classes of representations is finite.

The unitary irreducible representations $(\pi^\FF,\H^\FF)$ of $\FF$ in which we are interested are those which have positive energy in restriction to one of the subalgebras $\FF_{NS}$ or $\FF_R$, \cf \cite[Sect.3\&4]{Boc2} and also \cite{Boc}. In the first case, such a representation is a Fock space representation and unique and called the Neveu-Schwarz or vacuum representation of $\FF$, denoted by $\pi^\FF_{NS}$. In the second case, it is determined by an irreducible representation space of the 0-mode Clifford algebra generated by $\{F_0^a: a=1,...,d\}$. Such a space is of dimension $2^{[d/2]}$; if $d$ is even, it carries a natural grading and is unique; if instead $d$ is odd, there are two inequivalent spaces and they are ungraded, but their direct sum can be equipped with a natural grading, \cf \cite[Sect.5.3]{GBVF} together with \cite[Sect.3.12]{Was1} for details on the construction. We write $\pi^\FF_R$ for the unique (if $d$ is even) or one of the two (if $d$ is odd) irreducible representations, called Ramond representation of $\FF$.

For $(\pi^\FF,\H^\FF)$ either the Neveu-Schwarz or Ramond representation of $\FF$, we obtain a unitary representation of the $\lg$-current algebra at level $h^\vee$ by
\begin{equation}\label{eq:loop-JF}
 J^{a,\pi^\FF}_n:= \pi^\FF(J^a_n) := \frac12 \sum_r \sum_{b,c=1}^d f_{abc} \pi^\FF(F^b_{n-r})\pi^\FF(F^c_r).
\end{equation}
The summation in $r$ is obviously over the set $\frac12 + \Z$ for Neveu-Schwarz or $\Z$ for Ramond type, respectively. This defines the \emph{diagonal currents} $J_n^{a,\pi^\lg_{l,\lambda}\otimes\pi^\FF} = J_n^{a,\pi^\lg_{l,\lambda}}\otimes\unit_\FF+ \unit_{l,\lambda} \otimes J_n^{a,\pi^\FF}$ on $\H^\lg_{l,\lambda}\otimes\H^\FF$, which we shall frequently use henceforth. The vacuum diagonal currents are those determined by $J_n^{a,\pi_0}$, with $\pi_0=\pi^\lg_{l,0}\otimes\pi^\FF_{NS}$. In either of the representations $\pi=\pi^\lg_{l,\lambda}, \pi^\FF,\pi^\lg_{l,\lambda}\otimes\pi^\FF$, the $\lg$-current modes satisfy linear energy bounds in terms of the corresponding conformal Hamiltonian $L_0^\pi$, \cf \cite[Sect.2]{BSM} and \cite[Sect.4]{CW05}. In complete analogy to Example \ref{ex:superVir} this permits us to define unbounded selfadjoint smeared fields on (the Hilbert space) $\H_\pi$ localized in $I\in\I_0$ with invariant core $\Cci(L_0^\pi)$ as the closures
\[
 J^{\pi}(f)= \left( \sum_{n\in\Z} \sum_{a=1}^d f_{n.a} J^{a,\pi}_n \right)^{-},\quad f\in\Cci(\S,\lg)_I,
\]
where $f_{n,a}$, $n\in\Z$, denote the rapidly decreasing Fourier modes of the $e_a$-component of $f$. 
If $\pi=\pi^\lg_{l,\lambda}\otimes\pi^\FF$, then we write $F^\pi(f)\equiv \unit_{l,\lambda}\otimes F^{\pi^\FF}(f)$ and \eqref{eq:loop-JF} implies the following important commutation relations:
\begin{equation}\label{eq:loop-JFcomm}
[J^\pi(f),F^\pi(g)] = \rmi F^\pi([f,g]), \quad f,g\in \Cci(\S,\lg).
\end{equation}

For the Lie group $G$ and given level $l\in\N$, the \emph{current algebra net of $G$ at level $l$} is defined as
\[
 \A_{G_l}(I) := \{\rme^{\rmi J^{\pi^\lg_{l,0}}(fX)}:\; f\in\Cci(\S)_I, X\in \lg\}'', \quad I\in\I,
\]
which is often introduced in the equivalent way $\pi^G_{l,0}(\{ x\in \Loop G: x|_{I'}=\unit\})''$ and called \emph{loop group net}, with $\pi^G_{l,0}$ the integration to $\Loop G$ of $\pi^\lg_{l,0}$ on the Hilbert space $\H^\lg_{l,0}$, \cf \cite[Sect.3.9]{FG}, and \cite{PS86,Tol99,Was}. The \emph{$d$-fermion net} is the graded-local net defined as
\[
 \F(I) := \{F^{\pi^\FF_{NS}}(fX):\; f\in\Cci(\S)_I, X\in \lg\}'', \quad I\in \I,
\]
whose grading comes from $\FF$, \cf \cite{Boc,Boc} for the definition and structure of the even subnet. Both $\A_{G_l}$ and $\F$ are known to be diffeomorphism-covariant. Thus the tensor product net
\begin{equation}\label{eq:loop-def}
 \A := \A_{G_l}\otimes \F
\end{equation}
is a graded-local diffeomorphism-covariant net (with grading $\gamma$ the product of the trivial grading on the first and the nontrivial above grading on the second factor), which we call the \emph{super-current algebra net of $G$ at level $l + h^\vee$}. Its central charge is $c=\frac{d}{2} + \frac{dl}{l+h^\vee}$.

\begin{lemma}\label{lem:loop-diagonal}
The algebras $\A(I)$ are generated by the fermionic fields $F^{\pi^\FF_{NS}}(f)$ together with the exponentials of the diagonal currents $\rme^{\rmi J^{\pi^\lg_{l,0}\otimes\pi^\FF_{NS}}(f)}$, with $f\in \Cci(\S,\lg)_I$.
\end{lemma}

This follows immediately from the facts that $\rme^{\rmi J^{\pi^\lg_{l,0}\otimes\pi^\FF_{NS}}(f)}=\rme^{\rmi J^{\pi^\lg_{l,0}}(f)} \otimes \rme^{\rmi J^{\pi^\FF_{NS}}(f)}$ and that $\rme^{\rmi J^{\pi^\FF_{NS}}(f)}\in\F(I)$ as a consequence of \eqref{eq:loop-JFcomm} and graded Haag duality.

The super-Sugawara construction \cite[Sect.2\&4]{KT} in the representation $\pi=\pi^\lg_{l,\lambda}\otimes\pi^\FF$ gives now rise to a representation of the super-Virasoro algebra with central charge $c=\frac{d}{2} + \frac{dl}{l+h^\vee}$ and with generators
\begin{equation}\label{eq:free-Sugawara}
\begin{aligned}
 G_r^{\pi} :=&  \frac{1}{\sqrt{l+g}} \sum_{a,m} :\left( J^{a,\pi^G}_m + \frac13 J^{a,\pi^\F}_m \right) F^{a,\pi^\F}_{r-m}:\\ 
L_n^{\pi} :=& \frac{1}{2(l+g)} \sum_{a} :\left( \sum_m J^{a,\pi^G}_m J^{a,\pi^G}_{n-m} - \sum_r r F^{a,\pi^\F}_{r}F^{a,\pi^\F}_{n-r}\right):,
\end{aligned}
\end{equation}
where $:\cdot:$ stands for normal ordering, and $n\in\Z$ and $r\in\frac12 +\Z$ or $\Z$ (for $\pi^\FF$ of Neveu-Schwarz or Ramond type, respectively). Considering then for $\pi$ the (diagonal) vacuum representation $\pi_0=\pi^\lg_{l,0}\otimes\pi^\FF_{NS}$, we notice that the Lie algebra representation of the Virasoro algebra with generators $L_n^{\pi_0}$ integrates to the projective unitary representation $U$ of $\Diff^{(\infty)}$ \cite{Tol99} which turns out to be the one making the net $\A$ diffeomorphism-covariant. We would like to show that $\A$ is superconformal in the sense of Definition \ref{def:CFT-SCFTnet}. 

The procedure is standard (\cf \eg \cite[Sect.6.3]{CKL}), but for the reader's convenience we provide a sketch, and for the sake of readability we drop the superscripts $\pi$ which stand for the vacuum representation $\pi_0$ here. The (graded) commutation relations between the fields $L,G$ and $J,F$ in the above super-Sugawara construction are written in \cite[(2.5)]{KT}. In terms of smeared fields, a straight-forward computation yields (on the core $\Cci(L_0)$):
\begin{equation}
\begin{gathered}
\, [J(f),L(g)] = -\rmi J(f'g), \quad [F(f),L(g)] = -\rmi F(f'g) -\frac{\rmi}{2} F(fg'),\\
[J(f),G(g)] = -\rmi F(f'g), \quad [F(f),G(g)] = J(fg),
\end{gathered}
\end{equation}
for $f\in\Cci(\S,\lg)_I$ and $g\in\Cci(\S)_I$ with $I\in\I_\R$; notice that the first relation holds actually for every $I\in\I$. Thus, $L(g),G(g)$ (graded-) commute with $J(f),F(f)$ if $\supp f \cap \supp g = \emptyset$. Moreover, the four fields satisfy linear energy bounds w.r.t. $L_0$:
\begin{equation}\label{eq:loop-LEB}
\begin{gathered}
\| F(f) \xi \| = \frac{1}{\sqrt{2}} \|f\|_2 , \quad
\| J(f) \xi \| \le c_J(f) \|(\unit+L_0)\xi\|, \\
\| G(g) \xi \| \le c_G(g) \|(\unit+L_0)^{1/2}\xi\|, \quad
\| L(g) \xi \| \le c_L(g) \|(\unit+L_0)\xi\|,
\end{gathered}
\end{equation}
for all $\xi\in\Cci(L_0)$ and with suitable positive real constants $c_J(f),c_G(g),c_L(g)$ depending only on $f,g$, \cf \cite[(2.21)\&(2.23)]{BSM} and \cite[Sect.6.3]{CKL}. Following then \cite[Sect.6.3]{CKL} and applying \cite[Thm.3.1]{DF}, we see that $\rme^{\rmi L(f)}$ and $\rme^{\rmi G(f)}$ (graded-) commute with $\rme^{\rmi J(f)}$ and $F(f)$, so they lie in $Z\A(I)'Z^*=\A(I')$. They generate the super-Virasoro net 
$\A_{\SVir,c}$, which, by rotation covariance, satisfies
\[
 U(\Diff^{(\infty)}_I) \subset \A_{\SVir,c}(I) \subset \A(I), \quad I\in\I,
\]
and we can summarize the preceding discussion in

\begin{proposition}\label{def:loop-net}
The super-current algebra net of $G$ at level $l+h^\vee$, $\A = \A_{G_l}\otimes \F$, is a graded-local superconformal net.
\end{proposition}

The Ramond representation $(\pi^\FF_R,\H^\FF_R)$ of $\FF$ gives rise to an irreducible Ramond representation of the net $\F$ in the sense of Theorem \ref{th:CFT-RS-N}, \cf \cite{Boc,Boc2} together with \cite[Sect.6.4]{CKL}. 
We write $(\pR,\H_R)$, where 
$\H_R=\H_{l,0}\otimes\H^\FF_R$, for the corresponding irreducible Ramond representation of $\A$ which is the identity representation on the first component $\A_{G_l}$. As explained above, it is graded iff $d$ is even, and the eigenspace of minimal energy $\H_{R,0}$ has dimension $2^{[d/2]}$, where $[d/2]$ denotes the integral part of $d/2$. 
The restriction of $\pR$ to $\A^\gamma$ on $\H_R$ will be denoted again by $\pR$ (as in the preceding sections). According to Proposition \ref{prop:CFT-supersymmetric} and the explanation at the beginning of Section \ref{sec:gen-ST}, $\pR$ is \emph{supersymmetric} in the sense that it contains an (odd) operator $Q$ such that $Q^2= L_0^\pR - \lambda \unit$, namely $Q:=G_0^\pR$ with $\lambda=c/24$ (following the notation in Proposition \ref{prop:CFT-supersymmetric}). The trace-class condition  \eqref{eq:gen2-trace-class} can be shown in every locally normal representation of $\A$, in particular for $\pR$ and the vacuum representation, implying the split property of $\A$ \cite{DLR}. As can be seen in \cite[Sect.III.13]{Was1} or \cite{KT} or by a straight-forward computation, the following commutation relations hold on the invariant core $\Cci(L_0^\pi)$:
\begin{equation}\label{eq:loop-deltaJF}
 [Q,F^{\pR}(f)]_+ = \frac{1}{ \sqrt{l+h^\vee}} J^{\pR}(f),
 \quad [Q,J^{\pR}(f)]=  \rmi \sqrt{l+h^\vee} F^{\pR}(f'), \quad f\in\Cci(\S,\lg).
\end{equation}

Let us now come to the localized endomorphisms of $C^*(\A^\gamma)$.
Fix an interval $I_0\in\I$, let $\hat{I}_0'\subset \S^{(\infty)}$ (the latter regarded as multiplicative abelian group) denote the preimage of $I_0'$ under the universal covering map, and fix a smooth function $\hat{\phi}: \S^{(\infty)} \ra \R$ which is locally constant on $\hat{I}_0'$ and satisfies $\hat{\phi}(t\rme^{\rmi2\pi}) = \hat{\phi}(t) + 2\pi$. Given $z\in \ZZ(G)$, let $X_z\in\lg$ be a fixed choice such that $\exp(2\pi X_z) = z$, which is possible since $G$ is compact and connected whence $\exp$ is surjective; in particular, if $z=\unit$, we choose $X_z=0$. As $\ZZ(G)$ is the intersection of all maximal tori of $G$, $X_z$ lies actually in the maximal toral subalgebra $\mathfrak{t_g}\subset\lg$. For every $I\in\I$, let $\phi_{I}\in\Cci(\S)$ be a function with support in a proper interval of $\S$, coinciding with $\hat{\phi}$ modulo $2\pi\Z$ on $I$ (regarding $I$ as a subset of $\S^{(\infty)}$ via the universal covering map). Then the formula
\begin{equation}\label{eq:loop-rho-def}
(\pi_{l,z})_I(x) := \Ad\Big(\rme^{\rmi J^{\pi^\lg_{l,0}}(\phi_{I} X_z)}\Big)(x), \quad x\in \A_{G_l}(I),
\end{equation}
defines a representation, independent of the explicit choice of $X_z$ and $\phi_{I}$. It can be shown that it gives rise to a localized representation of the net $\A_{G_l}$ localized in $I_0$ and corresponding to  \cite[Sect.3.8\&3.9]{FG}. The equivalence class 
$[\pi_{l,z}]$ does not depend on the choice of the interval $I_0$ nor on the function $\hat{\phi}$ but only on $z \in \ZZ(G)$. 
Moreover, $\pi_{l,z}$ is equivalent to the vacuum representation if and only if $z$ is the neutral element of $G$.

A similar definition can be given on the component $\F$: Notice that
\[
\RR: f \in \K \mapsto \Ad(\exp( \hat{\phi} X_z)) (f) \in\K
\]
extends to a well-defined automorphism because $\hat{\phi}(t\rme^{\rmi2\pi}) = \hat{\phi}(t) + 2\pi$ and \linebreak $\exp(2\pi X_z) = z\in \ZZ(G)$ by construction. This enables us to define
\[
\pi_{\F,z}( F^{\pi^\FF_{NS}}(f)):= F^{\pi^\FF_{NS}}(\RR f) = F^{\pi^\FF_{NS}}( \Ad(\exp( \hat{\phi} X_z))( f)),
\quad f\in\K_I,
\]
for every $I\in\I$. Therefore,
\begin{align*}
\rme^{\rmi J^{\pi^\FF_{NS}}(\phi_{I} X_z)} F^{\pi^\FF_{NS}}(f) \rme^{- \rmi J^{\pi^\FF_{NS}}(\phi_{I} X_z)}
=& F^{\pi^\FF_{NS}}( \Ad(\exp(\phi_{I} X_z) )(f))\\ 
=& F^{\pi^\FF_{NS}}( \Ad(\exp(\hat{\phi} X_z)) (f))\\
=& \pi_{\F,z}( F^{\pi^\FF_{NS}}(f)),
\end{align*}
where the first line follows from a standard integration argument for the covariant $\Uone$-action on the fermionic currents (similar to \cite[Sect.6.3]{CKL}). Thus we obtain a representation $\pi_{\F,z}$ of $\F$, which restricted to $\A_{G_{h^\vee}}$ obviously has the form \eqref{eq:loop-rho-def}:
\[
(\pi_{\F,z})_I(x) = \Ad\Big(\rme^{\rmi J^{\pi^\FF_{NS}}(\phi_{I} X_z)}\Big)(x), \quad x\in \A_{G_{h^\vee}}(I), I\in\I.
\]
Let us write $\pi_z$ for the subrepresentation on the even subspace $\H^\Gamma\subset\H$ of the restriction of $\pi_{l,z}\otimes\pi_{\F,z}$ to $\A^\gamma$. It is clear from the definition that $\pi_z$ is actually a localized automorphism of the net $\A^\gamma$ localized in $I_0$.

For the sake of readability, we shall henceforth drop the superscripts $\pi_0$ on the fields in case $\pi$ is the vacuum (Neveu-Schwarz) representation $\pi_0$.

We shall consider two different types of localized endomorphisms of $C^*(\A^\gamma)$:
\begin{itemize}
\item[(1)] We have just constructed a family of localized automorphisms of $\A^\gamma=\A_{G_l}\otimes\F^\gamma$, namely the above $\pi_z$, with $z\in\ZZ(G)$ and localized in the given fixed $I_0$. We denote the corresponding localized endomorphisms of $C^*(\A^\gamma)$ by $\rho_z$. Since they have statistical dimension $1$ they are automorphisms, \cf Proposition \ref{prop:CFT-endoms}. Moreover, we have $[\rho_z]=[\rho_y]$ if and 
only if  $z=y$, and $[\rho_z \rho_y] = [\rho_{zy}]$. 

By construction their 2-variable cocycles (charge transporters), for given $I\in\I$ containing the closure of $I_0$, are the Weyl unitaries 
\[
z(\rho_z,g)=\iota_I\Big(\rme^{\rmi J((\hat{\phi}-\hat{\phi}\circ g)X_z)}\Big), \quad g\in \U_{I_0,I}.
\]
\item[(2)] As explained above, the irreducible unitary representations of $\hat{\Loop} \lg$ at level $l$ are $\pi^\lg_{l,\lambda}$ with 
$\lambda \in \Phi^\lg_l$ the corresponding integral dominant weight; their integration gives rise to a locally normal representation of $\A_{G_l}$, which can be implemented in the vacuum representation by an endomorphism localized in $I_0$, \cf \cite[Sect.3.8]{FG}. Tensoring with the vacuum representation of $\F^\gamma$, it defines a localized endomorphism of the product net $\A^\gamma$, and we write  $\rho_\lambda$ for the corresponding localized endomorphism of $C^*(\A^\gamma)$. 
\end{itemize}

We now define two subsets of $\Delta^0$ corresponding to the above two types of localized endomorphisms of $C^*(\A^\gamma)$, which we shall use in order to find applications of the main theorems in Section \ref{sec:gen-pairing}. Starting from those two, further examples can be constructed in a rather straight-forward manner left over to the reader.

\begin{definition}\label{def:loop-Delta}
Letting $\rho_z$ denote the endomorphisms of type (1) above, define 
\[
\tilde{\Delta} :=\textrm{semigroup generated by }\{\rho_z:z\in\ZZ(G)\}.
\]
Letting $\rho_\lambda$ denote the endomorphisms of type (2) above, define
\[
\Delta := \{\rho_\lambda: \lambda \in\Phi^\lg_l\}.
\]
\end{definition}

From the following proposition it shall become clear that $\tilde{\Delta} \subset \Delta^1$ in the sense of Definition \ref{def:gen2-Delta}. On the other hand, we cannot expect the inclusion $\Delta\subset\Delta^1$ to hold for all choices of the representatives $\rho_\lambda$ of type (2) above. Actually, we don't even know if this inclusion can hold for a suitable choice of those endomorphisms.

\subsection*{Derivation and spectral triples for the super-current algebra net}

Let us now investigate the derivation $(\delta,\dom(\delta))$ coming from the above ``supercharge" $Q=G_0^\pR$ in the Ramond representation $(\pR,\H_R)$:

\begin{proposition}\label{th:Loop-domain}
The set $\dom(\delta)\cap \pi_R( \A^\gamma (I))$ is a $\sigma$-weakly dense $*$-subalgebra of\linebreak $\pi_R(\A^\gamma(I))$ and hence 
$\pi_R^{-1}(\dom(\delta)) \cap \A^\gamma(I)$ is a $\sigma$-weakly dense $*$-subalgebra of $\A^\gamma(I)$, for all $I\in \I$.
 The elements $\rme ^{\rmi J(f X)}$ with $f  \in \Cci(\S)_I, X \in\lg$, lie in $\pi_R^{-1}(\dom(\delta)) \cap \A^\gamma(I)$, for all $I\in \I$.
In particular, the charge transporters of $\tilde{\Delta}$ are in $\pi_R^{-1}(\dom(\delta))$, so in fact $\tilde{\Delta}\subset \Delta^1$.
\end{proposition}

\begin{proof}
In order to understand the proof we notice that 
\[
(\pR)_I(\rme ^{\rmi J(fX)}) = \rme^{\rmi J^\pR(fX)},\quad f\in \Cci(\S)_I, I\in \I,
\]
and
\[
(\pR)_I(F(fX)) = F^\pR(fX),\quad f\in \Cci(\S)_I,I\in \I_\R,
\]
and we may therefore work with the expressions on the right-hand side. Moreover, in the first three steps of the present proof, we drop the superscripts $\pR$ for the sake of readability, so $J(fX)$, $F(fX)$, $L_0$ mean actually $J^\pR(fX)$, $F^\pR(fX)$ and $L_0^\pR$. 

For an arbitrary  interval $I\in\I_\R$, we shall prove that first the resolvents and then also the exponentials of $J^\pR(fX)$ lie in $\dom(\delta)\cap \pi_R( \A^\gamma (I))$, and in the third step we shall deal with the fermion fields. The fourth step concludes the proof.

(1) Recall that as always the subspace $\Cci(L_0)\subset\H_R$ is a common invariant core for all $J(f X)$, $F(fX)$, and $Q$, and  owing to the energy bounds \eqref{eq:loop-LEB}, these operators map $\dom(L_0^n)$ to $\dom(L_0^{n-1})$, for every $n\in\N$. Moreover, from \cite[Prop.4.3]{CHKL} we know that $(J(f X)-\lambda)^{-1}$ preserves the joint cores $\dom(L_0^n)$, for $n=1,2$, if $|\Im \lambda|$ is sufficiently large. Thus, for $\psi\in\dom(L_0^2)$, we have
\[
Q J(fX) (J(f X)-\lambda)^{-1}\psi = J(fX)Q (J(f X)-\lambda)^{-1}\psi + \rmi \sqrt{l+g}F(f' X) (J(f X)-\lambda)^{-1}\psi,
\]
by \eqref{eq:loop-deltaJF}. Adding $-\lambda Q(J(f X)-\lambda)^{-1}\psi$ on both sides and multiplying then by $(J(f X)-\lambda)^{-1}$ yields
\[
(J(f X)-\lambda)^{-1}Q\psi = Q (J(f X)-\lambda)^{-1}\psi +  (J(f X)-\lambda)^{-1} \rmi \sqrt{l+g} F(f' X) (J(f X)-\lambda)^{-1} \psi,
\]
so $(J(f X)-\lambda)^{-1} \in \dom (\delta)$ if $|\Im \lambda|$ is sufficiently large, and 
\[
\delta( (J(f X)-\lambda)^{-1}) =  - \rmi\sqrt{l+g} F(f' X) (J(f X)-\lambda)^{-2},
\]
using $[F(f' X), J(f X)]=0$, \cf \eqref{eq:loop-JFcomm}. This holds actually for every $\lambda\in\C\setminus\R$, which can be seen as follows: Suppose it holds for $\lambda_0$. The spectrum of $(J(f)-\lambda_0)^{-1}$ lies in $(\R-\rmi \Im\lambda_0)^{-1}$. Consider the complex map 
\[
\phi(z):=\frac{1}{z^{-1} +(\lambda_0- \lambda)}
\]
which is defined and analytic on an open neighborhood of $(\R-\rmi\Im\lambda_0)^{-1}\subset\C$ and has its only pole in $(\lambda-\lambda_0)^{-1}$. Holomorphic functional calculus then gives
\[
(J(f)-\lambda)^{-1} = \phi \Big( (J(f)-\lambda_0)^{-1}\Big) \in\dom(\delta),
\]
since $\dom(\delta)$ is closed under holomorphic functional calculus -- an adaptation of \cite[Prop.3.2.29]{BR}.

(2) By the same reasoning (with the spectral projections of $J(fX)$ denoted by\linebreak
$P_{J(fX)}(\cdot)$), also the exponentials $\rme ^{\rmi J(f X)}$ preserve $\dom(L_0^2)$. Using Borel functional calculus, Laplace transformation, step (1) and the selfadjointness of $Q$, $J(fX)$ and $F(f'X)$, we get, for all $\phi,\psi \in \dom(L_0^2)$,
\begin{align*}
\langle \phi, Q \rme ^{\rmi J(f X)}\psi \rangle 
=& \int_\R \rme ^{\rmi t} \rmd \langle Q\phi ,P_{J(fX)}(t)\psi\rangle\\ 
=& \int_\R \int_{\R+\rmi} \rme^{\rmi \lambda} (t-\lambda)^{-1} \rmd \lambda \rmd \langle Q\phi, P_{J(fX)}(t)\psi\rangle (t)\\
=&  \int_{\R+\rmi} \int_\R \rme^{\rmi \lambda} (t-\lambda)^{-1} \rmd \langle Q\phi ,P_{J(fX)}(t)\psi\rangle \rmd \lambda \\
= & \int_{\R+\rmi}  \rme^{\rmi \lambda} (\langle \phi, (J(f X)-\lambda)^{-1} Q \psi\rangle + \langle \phi, \delta((J(f X)-\lambda)^{-1} ) \psi\rangle) \rmd \lambda \\
=& \int_{\R+\rmi}  \rme^{\rmi \lambda} (\langle \phi, (J(f X)-\lambda)^{-1} Q \psi\rangle \rmd \lambda \\
&+  \int_{\R+\rmi} \rme^{\rmi \lambda} \langle \phi, \rmi \sqrt{l+g}F(f' X)(J(f X)-\lambda)^{-2} \psi\rangle ) \rmd \lambda \\
=&  \int_\R \int_{\R+\rmi}\rme^{\rmi \lambda} (t-\lambda)^{-1} \rmd \lambda \rmd \langle \phi, P_{J(fX)}(t) Q \psi\rangle \\
&+   \int_\R \int_{\R+\rmi}\rme^{\rmi \lambda}  (t-\lambda)^{-2} \rmd \lambda \rmd \langle \phi ,\rmi\sqrt{l+g} F(f'X) P_{J(fX)}(t)\psi\rangle \\
=&  \int_\R  \rme ^{\rmi t} \rmd \langle \phi, P_{J(fX)}(t) Q \psi\rangle
+  \int_\R \rme ^{\rmi t}\rmd \langle \phi ,\sqrt{l+g} F(f'X) P_{J(fX)}(t)\psi\rangle\\
=& \langle \phi,\rme ^{\rmi J(f X)} Q \psi\rangle +  \langle \phi, \sqrt{l+g} F(f' X) \rme ^{\rmi J(f X)} \psi\rangle 
\end{align*}
so $\rme ^{\rmi J(f X)} \in\dom (\delta)$ and $\delta(\rme ^{\rmi J(f X)}) =  \sqrt{l+g} F(f' X)\rme ^{\rmi J(f X)}$.

(3) Part (1) together with \eqref{eq:loop-deltaJF} shows only that 
\begin{equation}\label{eq:Loop-str-dense}
-s_1s_2 F(f_1X_1)(J(f_1X_1)+\rmi s_1)^{-1}F(f_2X_2)(J(f_2X_2)+\rmi s_2)^{-1}\in\dom(\delta),\quad s_1,s_2 \not= 0,
\end{equation}
\cf \cite{BG}.
To conclude the proof of denseness we would like to see that $F(f_1X_1)F(f_2X_2)$ is in the weak closure of the $*$-algebra generated by these operators. Using the spectral decomposition of the selfadjoint unbounded $J(fX)$, we obtain, for every $\psi\in\dom(L_0^2)$,
\[
 -\rmi s(J(fX)-\rmi s)^{-1} \psi = \int_\R \frac{-\rmi s}{t-\rmi s} \rmd P_{J(fX)} (t) \psi \ra \int_\R \rmd P_{J(fX)} (t) \psi = \psi,\quad s\ra \infty,
\]
by means of the dominated convergence theorem applied to the fact that $t\mapsto -\rmi s/(t-\rmi s)$ is bounded by the $P_{J(fX)}$-integrable function $t\mapsto 1$ and converges pointwise to 1, for $s\ra\infty$. Hence $-\rmi s (J(fX)-\rmi s)^{-1}\ra \unit$ strongly. Considering in the same way the limit $s_1,s_2\ra \infty$ in \eqref{eq:Loop-str-dense}, we infer that $F(f_1X_1)F(f_2X_2)$ lies in the strong closure of $\dom(\delta)$.

Thus the weak closure of the algebra generated by the elements
\[
\rme ^{\rmi J(f_1X_1)},\quad
F(f_1X_1)(J(f_1X_1)+\rmi)^{-1}F(f_2X_2)(J(f_2X_2)+\rmi)^{-1}
\]
in $\dom(\delta)$, with $f_i \in \Cci(\S)_I, X_i\in\lg$, coincides with $\pR(\A^\gamma(I))$ owing to Proposition \ref{prop:CFT-supersymmetric} and Lemma \ref{lem:loop-diagonal} -- and this proves the claimed denseness.

(4) Let us return to writing the superscripts ``$\pR$". So far we have the denseness of $\dom(\delta)\cap \pi_R( \A^\gamma (I)) \subset\pi_R(\A^\gamma(I))$, for $I\in\I_\R$. To obtain the statement for arbitrary $I\in\I$, we only have to use the rotation invariance of $\dom(\delta)$ together with covariance and local normality, which holds for the restriction of the general soliton $\pR$ to $\A^\gamma$. Concerning the explicit exponentials, we have proved that  $\rme ^{\rmi J^\pR(f X)}\in \dom(\delta) \cap \pR^{-1}(\A^\gamma(I))$, for $f \in \Cci(\S)_I, X\in\lg$ with $I\in\I_\R$, which is equivalent to saying $\rme ^{\rmi J(f X)}\in\pi_R^{-1}(\dom(\delta)) \cap \A^\gamma(I)$. Using rotation covariance, we obtain the statement actually for every $I\in\I$.
\end{proof}

We can now understand our above choice of the localized representations $\pi_z$: if instead of $\pi_{l,z}\otimes \pi_{\F,z}^\gamma$ we had taken $\pi_{l,z}\otimes \id_{\F^\gamma}$, then the cocycles would be $\rme^{\rmi J^{\pi^\lg_{l,0}}((\hat{\phi}-\hat{\phi}\circ g)X_z)}$, with $g\in {\mathcal U}_{I_0,I}$, which in general are not in $\pR^{-1}(\dom(\delta))$, so the corresponding automorphisms would not be differentiably transportable.

\subsection*{The index pairing for the super-current algebra net}

From the preceding proposition, Definition \ref{def:loop-Delta} and the theory of Section \ref{sec:gen-ST} we obtain two families of nontrivial $\theta$-summable spectral triples $(\AA_\Delta, (\pR\circ \rho,\H_R),Q)_{\rho\in\Delta}$ and $(\AA_{\tilde{\Delta}}, (\pR\circ \rho,\H_R),Q)_{\rho\in\tilde{\Delta}}$ over the locally convex algebras $\AA_\Delta$ and $\AA_{\tilde{\Delta}}$. They are even or odd, depending on whether $\pR$ is graded or ungraded, which in turn depends on whether $d$ is even or odd.
The pairing with K-theory needs several cases covered by Theorem \ref{th:gen2-pairing-even}(3) and \ref{th:gen2-pairing-odd}(2) for the algebra $\AA_{\tilde{\Delta}}$, and then by Theorem \ref{th:gen2-pairing-even}(2) and \ref{th:gen2-pairing-odd}(1) for the algebra $\AA_{\Delta}$:

\begin{theorem}\label{th:loop-even}
\begin{itemize}
\item[$(1)$] Suppose $d=\dim(G)$ is even. Then we have a family of even spectral triples $(\AA_{\tilde{\Delta}}, (\pR\circ \rho,\H_R),Q)_{\rho\in\tilde{\Delta}}$ with  $\AA_{\tilde{\Delta}}(I)= \pR^{-1}(\dom(\delta))\cap\A^\gamma(I)\subset \A^\gamma(I)$ \linebreak $\sigma$-weakly dense, for all $I\in\I$. Moreover, for $\rho,\sigma\in\tilde{\Delta}$, we have
\[
[\rho]=[\sigma] \quad\textrm{iff}\quad [\tau_{\rho}]=[\tau_{\sigma}].
\]
\item[$(2)$] Suppose $d$ is odd. Then we have a family of odd spectral triples $(\AA_{\tilde{\Delta}}, (\pR\circ\rho,\H_R),Q)_{\rho\in\tilde{\Delta}}$ with  $\AA_{\tilde{\Delta}}(I)= \pR^{-1}(\dom(\delta))\cap\A^\gamma(I)\subset \A^\gamma(I)$ $\sigma$-weakly dense, for all $I\in\I$. Moreover, for $\rho,\sigma\in\tilde{\Delta}$, we have
\[
[\rho]=[\sigma] \quad\textrm{iff}\quad [\tau_{\rho}]=[\tau_{\sigma}].
\]
\end{itemize}
For every $d$, the cocycles can be obtained as pullback cocycles $\tau_{\rho}=\rho^*\tau_{\id}$, for all $\rho\in\tilde{\Delta}$, since $\tilde{\Delta}$ forms a semigroup.
\end{theorem}

\begin{proof}
First recall that $\pi_R$ is graded iff $d$ is even.
The equality $\AA_{\tilde{\Delta}}(I)= \pR^{-1}(\dom(\delta))\cap\A^\gamma(I)$ follows from Proposition \ref{prop:gen2-locAA}(1) because $\tilde{\Delta}$ consists of differentiably transportable endomorphisms, and the denseness has been shown in Proposition \ref{th:Loop-domain}. This gives rise to spectral triples according to the general theory. Suppose $d$ is even. If $\rho,\sigma$ are inequivalent, then there are $y\not=z\in\ZZ(G)$ such that $\pR\circ\rho \simeq \pi_{l,z}\otimes \pi_1$ and $\pR\circ\sigma \simeq \pi_{l,y}\otimes \pi_2$, with $\pi_1,\pi_2$ two graded representations of $\F^\gamma$. Since $\pi_{l,z}$ and $\pi_{l,y}$ are irreducible and mutually inequivalent, $\pR\circ\rho$ and $\pR\circ\sigma$ must be disjoint. Since all elements in $\tilde{\Delta}$ are moreover automorphisms, we can apply Theorem \ref{th:gen2-pairing-even}(3) to obtain the complete separation of the cocycles corresponding to inequivalent endomorphisms. 

If $d$ instead is odd, then the representations $\pR\circ\rho$ are ungraded and hence the spectral triples odd. In that case we apply Theorem \ref{th:gen2-pairing-odd}(2) to obtain statement (2). The pullback statement becomes obvious using Remark \ref{rem:gen2-AA}(1).
\end{proof}

\begin{theorem}\label{th:loop-even2}
\begin{itemize}
\item[$(1)$] Suppose $d=\dim(G)$ is even. Then we have a family of even spectral triples $(\AA_\Delta, (\pR\circ \rho,\H_R),Q)_{\rho\in\Delta}$, and $\rho_\mu^{-1}(p_{0,+})\in \AA_\Delta$ if $\rho_\mu\in\Delta$ is an automorphism, and
\[
\tau_{\rho_\lambda}(\rho_\mu^{-1}(p_{0,+}))  = \dim (\H_{R,0,+})\delta_{\rho_\lambda,\rho_\mu}= 2^{d/2-1}\delta_{\lambda,\mu}.
\]
Hence, we can separate every entire cohomology class $[\tau_{\rho_\lambda}]$ from every $[\tau_{\rho_\mu}]$ with
$\rho_\mu\in\Delta$ an automorphism using the finite family 
\[
\{[\rho_\mu^{-1}(p_{0,+})] : \; \rho_\mu\in\Delta \; \textup{is an automorphism}\} \subset K_0(\AA_\Delta).
\]
\item[$(2)$] Suppose $d$ is odd. Then we have a family of odd spectral triples\linebreak $(\AA_\Delta, (\pR\circ \rho,\H_R),Q)_{\rho\in\Delta}$, and $\rho_\mu^{-1}(u_{0,+})\in \AA_\Delta$ if $\rho_\mu\in\Delta$ is an automorphism, and
\[
\tau_{\rho_\lambda}(\rho_\mu^{-1}(p_{0,+}))  = \dim (\H_{R,0,+})\delta_{\rho_\lambda,\rho_\mu}= 2^{(d-1)/2-1}\delta_{\lambda,\mu}.
\]
Hence, we can separate every cocycle $\tau_{\rho_\lambda}$ from every $\tau_{\rho_\mu}$ with $\rho_\mu\in\Delta$ an automorphism  using the finite family 
\[
\{[\rho_\mu^{-1}(u_{0,+})] : \; \rho_\mu\in\Delta \; \textup{is an automorphism}\} \subset K_1(\AA_\Delta).
\]
\end{itemize}
\end{theorem}

\begin{proof}
It follows directly from the definition that for $\lambda\not=\mu\in\Phi^\lg_l$, the two irreducible representations $\pR\circ\rho_\lambda$ and $\pR\circ\rho_\mu$ are disjoint. Thus we can apply Theorem \ref{th:gen2-pairing-even}(2) obtaining a separation of $\tau_{\rho_\lambda}$ from all $\tau_{\rho_\mu}$, $\rho_\mu\in\Delta$ an automorphism, by means of $\rho_\mu^{-1}(p_{0,+})$. Since by construction in Definition \ref{def:loop-Delta}, $\Delta$ contains no other endomorphism equivalent to $\rho_\lambda$, we obtain the final equivalence and separation statement of the theorem. The precise dimension of $\H_{R,0,+}$ is a consequence of the representation structure of the Clifford algebra of fermion 0-modes as explained above.

The odd case goes in complete analogy appealing to Theorem \ref{th:gen2-pairing-odd}(1).
\end{proof}

\subsection*{The super-Virasoro net and the case $\ind_{\H_{R,+}}Q \not= 0$}

Looking at Theorem \ref{th:gen2-pairing-even}, we would also like to study an example of case (1), where no differentiability condition is fulfilled and moreover the sectors are not automorphic but instead  $\ind_{\H_{R,+}}Q \not=0$. Consider the super-Virasoro net from Example \ref{ex:superVir} with $c=1$. We recall from \cite[Sect.6]{CKL} that in this case there is a unique graded irreducible Ramond representation $\pR$ of $\A_{\SVir,1}$ with lowest energy $h=\frac{1}{24}$ and $\ind_{\H_{R,+}}Q =1$. We have to check for which irreducible endomorphisms $\rho$ of $\A_{\SVir,1}^\gamma$ the representations $\pR\circ\rho$ and $\pR$ are disjoint.

In the notation of \cite[Sect.7]{CKL}, $\pR=\pi_{R+}\oplus\pi_{R-}$ corresponds to the representation $(121)_+\oplus (121)_-$ of the coset of $\operatorname{SU}(2)_4 \subset \operatorname{SU}(2)_2\otimes \operatorname{SU}(2)_2$ identified with our net $\A_{\SVir,1}^\gamma$. We first compute the S-matrix of $\A_{\SVir,1}^\gamma$ using the formulae stated there for the coset construction of $\A_{\SVir,1}^\gamma$, and then we compute the fusion matrices using the results for the S-matrix and the Verlinde formula. We would like to separate the sectors of $\A_{\SVir,1}^\gamma$ from the vacuum sector. We cannot expect this to be possible for all sectors, but we may seek those for which it is, \ie those localized endomorphisms $\rho$ for which the two representations $\pi_{R\pm}\circ\rho$ are disjoint from $\pi_{R+}\oplus \pi_{R-}$. In terms of the fusion matrix, this means the entries $N_{\pi_{R\pm},\rho}^{\pi_{R\pm}}$ have to be all $0$. Write $\Delta\setminus \{\id\}$ for the set of those localized endomorphisms $\rho$, assuming exactly one representative endomorphism per equivalence class. Identifying every sector $(jkl)$ (in the notation of \cite[Sect.7]{CKL}) with that representative and in particular $(000)$ with $\id\in\Delta$, explicit computations yield:
\begin{equation}\label{eq:loop-SVir}
\Delta:= \{ (000),\; (110),\; (130),\; (031),\; (121)_+,\; (121)_-,\; (141),\; (231) \}.
\end{equation}

Define the associated algebra $\AA_\Delta$ and the spectral triples $(\AA_\Delta,\pR\rho,G_0^\pR)_{\rho\in\Delta}$ as in Definition \ref{def:gen2-AA}. Theorem \ref{th:gen2-pairing-even}(1) with $\sigma=\id$ can now be applied yielding

\begin{proposition}\label{prop:loop-SVir}
For the net $\A_{\SVir,1}$, the irreducible graded Ramond representation $\pR$ with lowest energy $1/24$ and the endomorphism set $\Delta$ from \eqref{eq:loop-SVir}, we have a family of JLO cocycles ${\tau_\rho}$ associated to the even $\theta$-summable spectral triples 
$(\AA_\Delta,(\pR \circ \rho,\H_R),G_0^\pR)$, $\rho\in\Delta$. They can be separated from $\tau_{\id}$ as follows:
\[
{\tau_\rho} (s(\pR)) =  \left\lbrace\begin{array}{l@{\; :\; }l}
1 & \rho=\id\\
0 & \rho\not= \id .
\end{array}\right.
\]
\end{proposition}

We can also separate the cocycles associated to $\rho\in\Delta$ from those associated to any other $\sigma\in\Delta$ as proposed in Theorem \ref{th:gen2-pairing-even}(1), but it does not give further insight into the situation, wherefore we shall stop at this point. The present example should just serve as a simple illustration of the fact that, even in the case where we have no differentiability properties and no automorphic sectors but instead $\ind_{\H_{R,+}}Q \not=0$, we still get non-trivial noncommutative geometric (cohomology) invariants. 
Many further examples may be treated in a similar way in order to study more specific aspects.

\bigskip

All in all, this section of examples strongly confirms that the general construction of Section \ref{sec:gen-ST} and the various situations treated in Theorems \ref{th:gen2-pairing-even} and \ref{th:gen2-pairing-odd} show up naturally in well-known models of superconformal nets and help towards a completely new understanding of the latter in terms of noncommutative geometry.

\bigskip

\noindent
{\bf Acknowledgements.} We would like to thank Joachim Cuntz for useful explanations.

\end{document}